\documentclass[12pt]{article}
\usepackage{mathrsfs}
\usepackage{amsfonts,amscd,amsthm}
\usepackage{amsmath}
\usepackage{amssymb}
\usepackage{bbm}
\usepackage{graphicx}
\usepackage{enumerate}
\usepackage{geometry}
\usepackage{xcolor}
\usepackage{txfonts}
\usepackage{cite}
\usepackage{bbding}
\usepackage{comment}

\newtheorem{definition}{Definition}[section]
\newtheorem{theorem}{Theorem}[section]
\newtheorem{corollary}{Corollary}[section]
\newtheorem{lemma}{Lemma}[section]

\newtheorem{proposition}{Proposition}[section]
\newtheorem{remark}{Remark}[section]
\newtheorem{example}{Example}[section]
\newcommand{\B}{{\mathbb{B}}}

\newcommand{\R}{\mathbbm{R}}
\newcommand{\Sph}{{\mathbb{S}}}

\DeclareMathOperator*\epi{epi}

\DeclareMathOperator*\bdry{bdry}
\DeclareMathOperator*\argmax{arg\,max}

\DeclareMathOperator*\co{conv}
\DeclareMathOperator*\cl{cl}
\DeclareMathOperator*\pos{pos}
\DeclareMathOperator*\inte{int}
\DeclareMathOperator*\ri{ri}

\DeclareMathOperator*\dom{dom}

\DeclareMathOperator*\gph{gph}
\DeclareMathOperator*\lip{lip}

\DeclareMathOperator*\rg{rg}
\begin{document}
\begin{center}
{\Large\Large\sc {\bf Lipschitz-like property relative to a set and the generalized Mordukhovich criterion}}
\end{center}

\begin{center}
{\sc K.\ W.\ Meng}
\\ {\small School of Economic Mathematics, Southwest University of Finance and Economics, Chengdu 611130, China}\\
Email: mengkw@swufe.edu.cn\\[0.5cm]

{\sc M.\ H.\ Li}
\\ {\small School of Mathematics and Big Data, Chongqing University of Arts and Sciences, Yongchuan, Chongqing, 402160, China}\\
Email: minghuali20021848@163.com\\[0.5cm]

{\sc W.\ F.\ Yao}\\ {\small Department of Applied Mathematics, The Hong Kong Polytechnic University,  Hong Kong}\\
Email: dorothy.wf.yao@connect.polyu.hk\\[0.5cm]

{\sc X.\ Q.\ Yang}\\ {\small Department of Applied Mathematics, The Hong Kong Polytechnic University,  Hong Kong}\\
Email: mayangxq@polyu.edu.hk
\end{center}

\begin{center}
\end{center}

 \noindent{\bf Abstract}: In this paper we will establish some necessary condition and sufficient condition respectively for a set-valued mapping to have the Lipschitz-like property relative to a closed set by employing regular normal cone and limiting normal cone of a restricted graph of the set-valued mapping. We will obtain a complete characterization for a set-valued mapping to have the Lipschitz-property relative to a closed and convex set by virtue of the projection of the coderivative onto a tangent cone. Furthermore, by introducing a projectional coderivative of set-valued mappings, we establish a verifiable generalized Mordukhovich criterion for the Lipschitz-like property relative to a closed and convex set. We will study the representation of the graphical modulus of a set-valued mapping relative to a closed and convex set by using the outer norm of the corresponding projectional coderivative value. For an extended real-valued function, we will apply the obtained results to investigate its Lipschitz continuity relative to a closed and convex set and the Lipschitz-like property of a level-set mapping relative to a half line.

\noindent {\bf Keywords}: Lipschitz-like property relative to a set, projectional coderivative, Mordukhovich criterion, level-set mapping, graphical modulus.

\section{Introduction}
Stability theory of set-valued mappings has been extensively investigated. The monographs \cite{roc,bs00,mor,mor18,dr14} contain a comprehensive presentation for the derivation of conditions ensuring various stability properties of set-valued mappings, including the Lipschitz-like property.

The Lipschitz-like property of set-valued mappings (known also as Aubin property, or pseudo-Lipschitz property) has been introduced in \cite{a84}. The Lipschitz-like property has been well studied in the literature by virtue of Mordukhovich criterion, which was initially developed by \cite{mor03} and a more direct proof was given in \cite{roc} by using the basic variational analysis tools. One important assumption when Mordukhovich criterion is applied is that the candidate parameter under consideration is in the interior of the domain of set-valued mappings. It is worth noting that this criterion has been applied in the study of the solution mapping of generalized equations, linear semi-infinite and infinite systems and a parametric linear constraint system in \cite{levym04,clmp09,hy16} respectively. The Lipschitz-like property of the stationary set of some constrained optimization problems has been explored in \cite{levym04,ly14}. A critical face condition was developed in \cite{dr96} as necessary and sufficient conditions for the Lipschitz-like property of the solution mapping of a linear variational inequality problem and a nonlinear variational inequality problem (via linearization) over a polyhedral set.   

Metric regularity relative to a set was studied in \cite{ai06,io15,ht15} by using strong slopes, and directional metric (sub)regularity, isolated calmness and Lipschitz-like property relative to a set were explored in \cite{g13,bgo19} by using directional limiting coderivative respectively.

In this paper we will study the Lipschitz-like property relative to a set for set-valued mappings. When the restricted set is closed, we will obtain some necessary condition and sufficient condition respectively for set-valued mappings to have the Lipschitz-like property relative to the set by employing regular normal cone and limiting normal cone of a restricted graph of the set-valued mapping and the tangent cone of the set near the candidate point. When the set is closed and convex, we will obtain a complete characterization for a set-valued mapping to have the Lipschitz-property relative to the set by virtue of the projection of the coderivative onto the tangent cone. Furthermore, by employing an outer limit of the projection of the coderivative onto the tangent cone, we will introduce a projectional coderivative of set-valued mappings and apply it to establish a verifiable generalized Mordukhovich criterion for the Lipschitz-like property relative to the set. We will study the representation of the graphical modulus of a set-valued mapping relative to a closed and convex set by using the outer norm of the corresponding projectional coderivative value. {When the set is merely closed, we will compare our sufficient condition with the one that is derived via the directional limiting coderivative in \cite{bgo19}.}

For an extended real-valued function, the Lipschitz continuity relative to a closed set is equivalent to that its profile mapping has Lipschitz-like property relative to the set. By virtue of this equivalence, we will apply the obtained results to derive a full characterization for an extended real-valued function to have the Lipschitz continuity relative to a closed and convex set of the domain, including that of a sublinear function.
Moreover, a reinterpretation of subgradients of an extended real-valued function $f$ has been pointed out in \cite[Theorem 9.41]{roc} by applying the Mordukhovich criterion. That is, level-set mapping $S: \alpha \mapsto  \{x | f(x) - \langle \bar{v}, x-\bar{x}\rangle \leq \alpha\}$ fails to have the Lipschitz-like property if and only if $\bar{v} \in \partial f(\bar{x})$.
Whenever $f$ is convex and $\bar{v} \in \partial f(\bar{x})$, $\bar{x}$ is clearly a minimum of $f(x) - \langle \bar{v}, x-\bar{x}\rangle$ implying that $\mbox{dom} S = \{ \alpha | \alpha \geq f(\bar{x})\}$ and hence that $f(\bar{x}) \notin \mbox{int(dom} S)$.
By virtue of the generalized Mordukhovich criterion, we will show that $S$ fails to have the Lipschitz-like property relative to $\mbox{dom} S$ at $f(\bar{x})$ for $\bar{x}$ if and only if $\bar{v}$ belongs to $\partial^>_{\bar{v}} f(\bar{x})$,  the outer limiting subdifferential set of $f$ at $\bar{x}$ with respect to $\bar{v}$, see its definition in Section 4. In the case of $\bar{v}=0$, $\partial^>_{\bar{v}} f(\bar{x})$ reduces to the outer limiting subdifferential set $\partial^> f(\bar{x})$, which has been studied extensively  for the study of error bounds in the literature, see  \cite{io08, knt10,fhko10,io15,ca2, lmy2017, ers2019}.

The organization of the paper is as follows. In Section 2, we will employ regular normal cone and limiting normal cone of a restricted graph of the set-valued mapping to obtain some neighborhood necessary condition and sufficient condition respectively for set-valued mappings to have the Lipschitz-like property relative to a closed set. We will introduce a projectional coderivative of set-valued mappings and apply it to establish a verifiable generalized Mordukhovich criterion for the Lipschitz-like property relative to a closed and convex set. In Section 3, we will obtain characterization of the Lipschitz continuity relative to a closed and convex set of an extended real-valued function. In Section 4, we will apply the obtained results to investigate the Lipschitz-like property relative to the half line for a level-set mapping.

Throughout the paper we use the standard notations of variational analysis; see the seminal book \cite{roc} by Rockafellar and Wets.   For readers' convenience, we also mention alternative names for some of notions used in this paper, see \cite{mor}.
The Euclidean norm of a vector $x$ is denoted by $||x||$, and the inner product of vectors $x$ and $y$ is denoted by $\langle x, y\rangle$. We denote by $[x]^\perp:=\{v\in \R^n\mid \langle v, x\rangle=0\}$ the orthogonal space of the vector $x$.
 Let $\B$ denote the closed unit Euclidean ball and let $\Sph$ denote the unit sphere. We denote by $\B_\delta(x)$ the closed ball centered at $x$ with radius $\delta>0$.

Let $A\subset \R^n$ be a nonempty set. We say that $A$ is locally closed at a point $x\in A$ if $A\cap U$ is closed for some closed neighborhood $U$ of $x$. We denote the interior,  the relative interior, the closure,  the boundary, the convex hull and the positive hull of $A$ respectively by $\inte A$,  $\ri A$, $\cl A$, $\bdry A$, $\co A$  and $\textup{pos}A:=\{0\}\cup \{\lambda x| x\in A$ and $\lambda>0\}$. We denote by $A^\perp:=\{v\in \R^n\mid \langle v, x\rangle=0\;\forall x\in A\}$ the orthogonal space of $A$, and by
\[
A^\infty:=\{x\in \R^n\mid \exists x_k\in A,\,\lambda_k\downarrow 0,\,\mbox{with}\,\lambda_kx_k\to x\}
\]
the horizon cone of $A$.  The polar cone of $A$ is defined by
\[
A^*:=\{v\in \R^n\mid \langle v, x\rangle\leq 0\;\forall x\in A\}.
\]
The support function $\sigma_A$  of $A$ is defined by
\[
\sigma_A(x):=\sup_{v\in A}\langle v, x\rangle.
\]
The indicator function $\delta_A$ of $A$ is defined by
\[
\delta_A(x):=\left\{
\begin{array}{ll}
0 &\mbox{if}\;x\in A,\\
+\infty &\mbox{otherwise}.
\end{array}
\right.
\]
The distance from $x$ to $A$ is defined by
\[
d(x,A):=\inf_{y\in A}||y-x||.
\]
The projection mapping ${\rm proj}_A$ is defined by
\[
{\rm proj}_A(x):=\{y\in A\mid \|y-x\|=d(x,A)\}.
\]
For a set $M\subset \R^n$, we denote the projections of $M$ on $A$  by
\[
{\rm proj}_AM:=\{y\in A\mid \exists x\in M,\,\mbox{with}\, \|y-x\|=d(x,A)\}.
\]
If $A=\emptyset$, we use the convention that $d(x, A):=+\infty$, ${\rm proj}_A(x):=\emptyset$, and ${\rm proj}_AM:=\emptyset$.
The excess of $A$ over another  nonempty  set $B\subset \R^n$ is defined by
\[
e(A,B):=\sup\{d(x, B)\mid x\in A\},
\]
with  the convention  that $e(A, B):=0$ for $A=\emptyset$.

Let $x\in A$. We use $T_A(x)$ to denote the  tangent/contingent cone to $A$ at $x$, i.e. $w\in T_A(x)$ if there exist sequences $t_k\downarrow 0$ and $\{w_k\}\subset \R^n$ with $w_k\rightarrow w$ and $x+ t_k w_k\in A \;\forall k$.
We denote by $N_A^{\rm prox}(x)$ the proximal normal cone to $A$ at $x$, i.e., $v\in N_A^{\rm prox}(x)$ if there exists some $t>0$ such that $x\in {\rm proj}_A(x+tv)$.
The regular/Fr\'{e}chet  normal cone, $\widehat{N}_A(x)$ to $A$ at $x$ is the polar cone of $T_A(x)$. A vector $v\in \R^n$ belongs to the (basic/limiting/Mordukhovich) normal cone $N_A(x)$ to $A$ at $x$, if there exist sequences $x_k\to x$ and $v_k\to v$ with $x_k\in A$ and $v_k\in \widehat{N}_A(x_k)$ for all $k$.   It is well known that
\[
N_A^{\rm prox}(x)\subset \widehat{N}_A(x)\subset N_A(x).
\]
$A$ is said to be regular at $x$ in the sense of Clarke if it is locally closed at $x$ and $\widehat{N}_A(x)=N_A(x)$.

Let $C\subset \R^n$ be a nonempty convex set.  A face of   $C$ is a convex subset $C'$ of $C$ such that every closed line segment in $C$ with a relative interior point in $C'$ has both endpoints in $C'$.  An exposed face of $C$ is the intersection of $C$ and a non-trivial supporting hyperplane to $C$.  In other words, $F$ is an exposed face of $C$ if and only if there is some $x\in \R^n$ such that $F=\argmax_{v\in C}\langle x, v\rangle$.  See the book \cite{rock70} for more details.

For a set-valued mapping $S:\R^n\rightrightarrows \R^m$, we denote by
\[
\gph S:=\{(x, u)\mid u\in S(x)\}\quad\mbox{and}\quad \dom S:=\{x\mid S(x)\not=\emptyset\}
\]
  the graph   and the domain of $S$, respectively. $S$ is said to be positively homogeneous if
  \[
  0\in S(0)\quad\mbox{and}\quad S(\lambda x)=\lambda S(x)\quad\mbox{for all}\; \lambda>0\;\mbox{and}\; x,
  \]
  or in other words, $\gph S$ is a cone. If $S$ is a positively homogeneous mapping, the outer norm  of $S$ is denoted and defined by
  \[
  |S|^+:=\sup_{x\in \B}\sup_{u\in S(x)}\|u\|,
  \]
  which is the infimum over all  constants $\kappa\geq 0$ such that $\|u\|\leq \kappa\|x\|$ for all pairs $(x, u)\in \gph S$.

   Consider a point $\bar{x}\in \dom S$. The outer limit of $S$ at $\bar{x}$ is defined by
\[
\limsup_{x\to \bar{x}}S(x):=\{u\in \R^m| \exists x_k\to \bar{x}, \exists u_k\to u\;\textup{with}\;u_k\in S(x_k)\}.
\]
$S$ is said to be outer semicontinuous  at $\bar{x}$ if
\[
\limsup_{x\to \bar{x}}S(x)\subset S(\bar{x}).
\]
   The  (normal) coderivative, the regular/Fr\'{e}chet coderivative and the proximal coderivative of $S$ at $\bar{x}$ for any $\bar{u}\in S(\bar{x})$ are respectively the mapping $D^*S(\bar{x}\mid \bar{u}):\R^m\rightrightarrows \R^n$ defined by
  \[
  x^*\in D^*S(\bar{x}\mid \bar{u})(u^*)\Longleftrightarrow (x^*, -u^*)\in N_{\gph S}(\bar{x}, \bar{u}),
  \]
  the mapping $\widehat{D}^*S(\bar{x}\mid \bar{u}):\R^m\rightrightarrows \R^n$ defined by
  \[
  x^*\in \widehat{D}^*S(\bar{x}\mid \bar{u})(u^*)\Longleftrightarrow (x^*, -u^*)\in \widehat{N}_{\gph S}(\bar{x}, \bar{u}),
  \]
  and
  the mapping $D^{\rm *prox}S(\bar{x}\mid \bar{u}):\R^m\rightrightarrows \R^n$ defined by
  \[
  x^*\in D^{\rm *prox}S(\bar{x}\mid \bar{u})(u^*)\Longleftrightarrow (x^*, -u^*)\in N^{\rm prox}_{\gph S}(\bar{x}, \bar{u}).
  \]
  Clearly, the following inclusions hold:
  \[
  \gph D^{\rm *prox}S(\bar{x}\mid \bar{u})\subset\gph \widehat{D}^*S(\bar{x}\mid \bar{u})\subset \gph D^*S(\bar{x}\mid \bar{u}).
   \]
  For a set $X\subset \R^n$, we denote by
  \[
  S|_X:=\left\{
  \begin{array}{ll}
    S(x) & \mbox{if}\;x\in X, \\[0.2cm]
    \emptyset & \mbox{if}\;x\not\in X,
  \end{array}
    \right.
  \]
  the restricted mapping of $S$ on $X$. It is clear to see that
  \[
  \gph S|_X=\gph S\cap (X\times \R^m)\quad\mbox{and}\quad \dom  S|_X=X\cap \dom S.
  \]

Let $g:\R^n\to \overline{\R}:=\R\cup\{\pm \infty\}$ be an extended real-valued function and let $\bar{x}$ be  a point with $g(\bar{x})$ finite.  We denote the epigraph and the domain of $g$ by
\[
\epi g:=\{(x,\alpha)\mid g(x)\leq \alpha\}\quad\mbox{and}\quad \dom g:=\{x\mid g(x)<+\infty\},
\]
respectively.  $g$ is said to be locally lower semicontinuous (for short, lsc) at $\bar{x}$, if there is an $\epsilon > 0$ such that
all sets of the form $\{x\mid \|x-\bar{x}\|\leq \epsilon,\,g(x)\leq \alpha\}$ with $\alpha\leq g(\bar{x})+\epsilon$ are closed, see \cite[Definition 1.33]{roc}.  It is well-known that $g$ is locally lsc at $\bar{x}$ if and only if $\epi g$ is locally closed at $\left(\bar{x}, g(\bar{x})\right)$.

The vector $v\in \R^n$ is a regular/Fr\'{e}chet subgradient of $g$ at $\bar{x}$, written $v\in \widehat{\partial} g(\bar{x})$, if
\[
g(x)\geq g(\bar{x})+\langle v, x-\bar{x}\rangle+o(||x-\bar{x}||).
\]
The vector $v\in \R^n$ is a (general/basic) subgradient of $g$ at $\bar{x}$, written $v\in \partial g(\bar{x})$, if there exist sequences $x_k\to \bar{x}$ and $v_k\to v$ with $g(x_k)\to g(\bar{x})$ and $v_k\in \widehat{\partial} g(x_k)$.  The subdifferential set $\partial g(\bar{x})$ is also referred to as limiting/Mordukhovich subdifferential.   The vector $v\in \R^n$ is a horizon/singular subgradient of $g$ at $\bar{x}$, written $v\in \partial^\infty g(\bar{x})$, if there are sequences $x_k\to \bar{x}$ with $g(x_k)\to g(\bar{x})$, $\lambda_k\downarrow 0$ and $v_k\in \widehat{\partial} g(x_k)$ such that $\lambda_kv_k\to v$.

The outer limiting subdifferential of $g$ at $\bar{x}$ introduced in \cite{io08}  is denoted and defined as follows:
    \[
\partial^>  g(\bar{x}):=\left\{\lim_{k\rightarrow +\infty}v_k\mid \exists x_k\rightarrow_g \bar{x},\;\forall k:\,g(x_k)>g(\bar{x})\,\mbox{and}\,v_k\in \partial g(x_k)\right\},
 \]
 which coincides with the outer limiting subdifferential set $\partial^>_{0}g(\bar{x})$ of $g$ at $\bar{x}$ with respect to $0$,  a notion newly introduced in  Section \ref{sec-level-set and ss}. Note that a closely related notion, called the right-sided subdifferential, was given in \cite{mor05} and defined by using a weak inequality instead of the strict inequality used here. See  \cite[Definition 1.100 and Theorem 1.101]{mor} for more details on  the right-sided subdifferential and its applications.

\section{The general case}
In this section, we consider a  set-valued mapping $S: \R^n\rightrightarrows  \R^m$, a pair $(\bar{x}, \bar{u})\in \gph S$  and a set $X\subset\R^n$ with $\bar{x}\in X$. First, we formally give the definition  for the Lipschitz-like property of $S$ relative to $X$ at $\bar{x}$ for $\bar{u}$ and accordingly  the definition for the graphical modulus $\lip_X S(\bar{x}\mid\bar{u})$. In the case of $X$ being closed,  by using the (regular) coderivative  of the restricted mapping $S|_X$,  we then present a (necessary) sufficient condition  for the relative Lipschitz-like property. Finally by filling the gap between these necessary and sufficient conditions, we will focus on the case that $X$ is closed and convex. We give two characterizations for the relative Lipschitz-like property and two corresponding formulas for the graphical modulus $\lip_X S(\bar{x}\mid\bar{u})$ when $X$ is closed and convex: one is a uniform boundedness condition getting all nearby points involved, and the other one is a point-based condition based on the so-called projectional coderivative, a newly introduced notion. In both of these characterizations, the projections of the coderivative  of  the restricted mapping $S|_X$ onto the tangent cones to $X$ play a key role.

The definition below is borrowed from \cite[Definition 9.36]{roc}.
\begin{definition}[Lipschitz-like property relative to a set]\label{def-lip-like}
A mapping $S: \R^n\rightrightarrows  \R^m$ has the Lipschitz-like property relative to $X$ at $\bar{x}$ for $\bar{u}$, where $\bar{x}\in X$ and $\bar{u}\in S(\bar{x})$, if $\gph S$ is locally closed at $(\bar{x}, \bar{u})$ and there are neighborhoods $V\in \mathcal{N}(\bar{x})$, $W\in \mathcal{N}(\bar{u})$, and a constant $\kappa \in \R_+$ such that
\begin{equation}\label{just-def}
S(x')\cap W\subset S(x)+\kappa\|x'-x\|\B\quad \forall x,x'\in X\cap V.
\end{equation}
The graphical modulus of $S$ relative to $X$ at $\bar{x}$ for $\bar{u}$ is then
\[
\begin{array}{ll}
\lip_X S(\bar{x}\mid\bar{u}):=\inf\;\{\;\kappa\geq 0&\mid  \exists V\in \mathcal{N}(\bar{x}), W\in \mathcal{N}(\bar{u}),\;\mbox{such that}\\[0.5cm]
&S(x')\cap W\subset S(x)+\kappa\|x'-x\|\B\quad \forall x,x'\in X\cap V\;\}
\end{array}
\]
\end{definition}

In the case of $X$ being  closed, we first  present a necessary condition for the Lipschitz-like property relative to $X$ by using the tangent cone $T_X(x)$ and the regular  coderivatives $\widehat{D}^*S|_X(x\mid u)$ of $S|_X$ for all nearby points $(x, u)$ of $(\bar{x}, \bar{u})$ in $\gph S|_X$.
\begin{theorem}[Necessity]\label{theo-nece}
Consider $S: \R^n\rightrightarrows  \R^m$, $\bar{x}\in X\subset \R^n$, $\bar{u}\in S(\bar{x})$ and $\kappa\geq 0$. Suppose that   $X$ is closed.
If $S$ has the Lipschitz-like property relative to $X$ at $\bar{x}$ for $\bar{u}$ with constant $\kappa$,   then  the condition
\begin{equation}\label{key-youjie}
\max_{w\in T_{X}(x)\cap \Sph}\langle x^*, w\rangle \leq \kappa \|u^*\|\quad\forall x^*\in \widehat{D}^*S|_X(x\mid u)(u^*)
\end{equation}
holds  for all $(x,u)$ close enough to $(\bar{x}, \bar{u})$ in $\gph S|_X$.
 \end{theorem}
\begin{proof} As $S$ has the Lipschitz-like property relative to $X$ at $\bar{x}$ for $\bar{u}$ with constant $\kappa$,   there exist   some  neighborhoods $V\in \mathcal{N}(\bar{x})$ and $W\in \mathcal{N}(\bar{u})$  such that
\begin{equation}\label{zkl}
S(x')\cap W\subset S(x)+\kappa\|x'-x\|\B\quad \forall x,x'\in X\cap V.
\end{equation}
 Without loss of generality, we can assume that the sets $V$ and $W$ are open. Let $x^*\in \widehat{D}^*S|_X(x\mid u)(u^*)$ (i.e., $(x^*,-u^*)\in \widehat{N}_{\gph S\cap (X\times \R^m)}(x,u)$) with $(x,u)\in \gph S|_X\cap (V\times W)$, and let $w\in T_X(x)\cap \Sph$. By the definition of tangent cone,  there exists  some $\{x_k\}\subset X\backslash\{x\}$ such that $x_k\rightarrow x$ and
\begin{equation}\label{kaixin}
\frac{x_k-x}{\|x_k-x\|}\rightarrow w.
\end{equation}
Clearly, there is some $k_0$ such that  $x_k\in X\cap V$ for all  $k\geq k_0$.   This, together with the facts that $x\in X\cap V$ and  $u\in S(x)\cap W$,  implies by (\ref{zkl}) the existence of $u_k\in S(x_k)$  such that
\begin{equation}\label{mantou}
\|u_k-u\|\leq \kappa \|x_k-x\|\quad  \forall k\geq k_0.
\end{equation}
Clearly, we have $u_k\rightarrow u$. In view of $(x^*,-u^*)\in \widehat{N}_{\gph S|_X}(x,u)$, we get from the definition of regular normal cone (cf. \cite[Defition 6.3 or 6(5)]{roc}) that
\begin{equation}\label{jianhua}
\lim_{k\rightarrow +\infty}\frac{\max\{\langle(x^*, -u^*), (x_k-x, u_k-u)\rangle, 0\}}{\|x_k-x\|+\|u_k-u\|}=0.
\end{equation}
By (\ref{mantou}), we have
\[
\frac{\max\{\langle(x^*, -u^*), (x_k-x, u_k-u)\rangle, 0\}}{\|x_k-x\|+\|u_k-u\|}\geq \frac{\max\{\langle x^*, x_k-x\rangle-\kappa\|u^*\|\|x_k-x\|, 0\}}{(1+\kappa)\|x_k-x\|}\quad   \forall k\geq k_0,
\]
which together with (\ref{kaixin}) and (\ref{jianhua}) implies the inequality
$ \langle x^*, w\rangle \leq \kappa \|u^*\|$ and hence (\ref{key-youjie}). This completes the proof. \end{proof}

In the case of $X$ being closed, we now present a sufficient condition for the Lipschitz-like property relative to $X$ by using the closure of the generated cone  $ \cl\pos (X-x)$ and the  coderivatives $D^*S|_X(x\mid u)$ of $S|_X$ for all nearby points $(x, u)$ of $(\bar{x}, \bar{u})$ in $\gph S|_X$. In our proof,  the Ekeland's variational principle plays a key role.
\begin{theorem}[Sufficiency]\label{theo-suff}
Consider $S: \R^n\rightrightarrows  \R^m$, $\bar{x}\in X\subset \R^n$, $\bar{u}\in S(\bar{x})$ and $\tilde{\kappa}>\kappa>0$. Suppose that $\gph S$ is locally closed at $(\bar{x}, \bar{u})$ and that $X$ is closed.   If the  condition
\begin{equation}\label{key-youjie000}
\max_{w\in \cl\pos (X-x)\cap \Sph}\langle x^*, w\rangle \leq \kappa \|u^*\|\quad \forall x^*\in D^*S|_X(x\mid u)(u^*)
\end{equation}
holds for all  $(x,u)$ close enough to $(\bar{x}, \bar{u})$ in $\gph S|_X$, then $S$ has the Lipschitz-like property relative to $X$ at $\bar{x}$ for $\bar{u}$ with  constant $\tilde{\kappa}$.
\end{theorem}
\begin{proof} Observing that  all the properties involved depend on the nature of $\gph S$ in an arbitrary small neighborhood of $(\bar{x}, \bar{u})$, there's  no harm, therefore, in assuming from now on that $\gph S$ is closed
in its entirety.

Let $0<\varepsilon<  \frac{\tilde{\kappa}-\kappa}{4\tilde{\kappa}}$.  Suppose by contradiction that $S$ does not have the Lipschitz-like property relative to $X$ at $\bar{x}$ for $\bar{u}$ with constant $\tilde{\kappa}$, meaning that there exist $x', x''\in \B_\varepsilon(\bar{x})\cap X$ with $x'\not=x''$, and $u''\in S(x'')\cap \B_\varepsilon(\bar{u})$ such that
\begin{equation}\label{fanmian}
d(u'', S(x'))>\tilde{\kappa}\|x''-x'\|:=\beta.
\end{equation}
Clearly, we have $0<\beta\leq 2\tilde{\kappa}\varepsilon$.

Define $\varphi:\R^n\times \R^m\rightarrow\R\cup\{+\infty\}$ by
\[
\varphi(x,u):=\|x-x'\|+\delta_{\gph S|_X}(x,u).
\]
Clearly,  $\varphi$ is lsc (due to  closedness of $\gph S$ and $X$) with $\inf \varphi$ being finite,
and
\[
\varphi(x'', u'')\leq \inf  \varphi+\frac{\beta}{\tilde{\kappa}}.
\]
By equipping the product space $\R^n\times \R^m$ with a norm $p$ defined by
\[
p(x,u):=\beta\|x\|+\|u\|,
\]
we   apply the Ekeland's variational principle to obtain some $(\tilde{x},\tilde{u})\in \R^n\times \R^m$ such that
\begin{equation}\label{per-dec}
p(\tilde{x}-x'', \tilde{u}-u'')\leq \frac{\kappa+\tilde{\kappa}}{2}\frac{\beta}{\tilde{\kappa}},
\end{equation}
\begin{equation}\label{per-fun}
 \varphi(\tilde{x}, \tilde{u})\leq\varphi(x'', u''),
\end{equation}
\begin{equation}\label{per-opt}
\displaystyle\arg\min_{x,\;u}\left\{\varphi(x, u)+\frac{2}{\kappa+\tilde{\kappa}}p(x-\tilde{x}, u-\tilde{u})\right\}=\{(\tilde{x}, \tilde{u})\}.
\end{equation}
From  (\ref{per-fun}), it follows that
\begin{equation}\label{near-gph}
(\tilde{x}, \tilde{u})\in \gph S\cap(X\times \R^m)=\gph S|_X
\end{equation}
and hence that
\[
\|\tilde{x}-x'\|\leq \|x''-x'\|.
\]
Then by the triangle inequality, we have
\begin{equation}\label{near-x}
\|\tilde{x}-\bar{x}\|\leq \|\tilde{x}-x'\|+\|x'-\bar{x}\|\leq \|x''-x'\|+\|x'-\bar{x}\|\leq \|x''-\bar{x}\|+2\|x'-\bar{x}\|\leq 3\varepsilon.
\end{equation}
From (\ref{per-dec}), it follows that \[\|\tilde{u}-u''\|\leq \frac{\kappa+\tilde{\kappa}}{2}\frac{\beta}{\tilde{\kappa}}<\beta\leq 2\tilde{\kappa}\varepsilon\] and
hence by the triangle inequality  that
\begin{equation}\label{near-u}
\|\tilde{u}-\bar{u}\|\leq \|\tilde{u}-u''\|+\|u''-\bar{u}\|\leq (2\tilde{\kappa}+1)\varepsilon.
\end{equation}
 So we have $\tilde{x}\not= x'$, for otherwise we have \[
 d(u'', S(x'))=d(u'', S(\tilde{x}))\leq \|\tilde{u}-u''\|<\beta,\]
  contradicting to (\ref{fanmian}). From (\ref{per-opt}) and the generalized
version of Fermat's rule \cite[Theorem 10.1]{roc}, it follows that
\begin{equation}\label{feierma}
(0,0)\in \partial(\psi+\delta_{\gph S|_X})(\tilde{x}, \tilde{u}),
\end{equation}
where
\[
\psi(x,u):=\|x-x'\|+\frac{2}{\kappa+\tilde{\kappa}}\displaystyle\left(\beta\|x-\tilde{x}\|+\|u-\tilde{u}\|\right).
\]
Clearly, $\psi$ is convex and Lipschitz continuous and in terms of closed unit balls $\B_1$ in $\R^n$ and $\B_2$ in $\R^m$,
\begin{equation}\label{jiandanjisuan}
\partial\psi (\tilde{x}, \tilde{u})=\left(\frac{\tilde{x}-x'}{\|\tilde{x}-x'\|}+\frac{2\beta}{\kappa+\tilde{\kappa}}\B_1\right)\times \frac{2}{\kappa+\tilde{\kappa}}\B_2.
\end{equation}
Applying the calculus rule for subgradients of Lipschitzian sums \cite[Exercise 10.10]{roc}, we deduce from (\ref{feierma}) that
\[
(0, 0)\in \partial\psi (\tilde{x}, \tilde{u})+N_{\gph S|_X}(\tilde{x}, \tilde{u}).
\]
This, together with (\ref{jiandanjisuan}), implies the existence of   $v_1\in \B_1$, $v_2\in \B_2$ and
\begin{equation}\label{near-normal}
(x^*, -u^*)\in N_{\gph S|_X}(\tilde{x}, \tilde{u})\Longleftrightarrow x^*\in D^*S|_X(\tilde{x}\mid \tilde{u})(u^*)
\end{equation}
 such that
\[
x^*=-\frac{\tilde{x}-x'}{\|\tilde{x}-x'\|}-\frac{2\beta}{\kappa+\tilde{\kappa}}v_1,
\]
and
\[
u^*=\frac{2}{\kappa+\tilde{\kappa}}v_2.
\]
Since   $\tilde{x},x'\in X$ with $\tilde{x}\not=x'$, we have
\[
w^*:=\frac{x'-\tilde{x}}{\|x'-\tilde{x}\|}\in \cl\pos(X-\tilde{x})\cap \Sph.
\]
Then we have
\[
\begin{array}{lll}
\langle x^*, w^*\rangle-\kappa \|u^*\|&=&1-\frac{2\beta}{\kappa+\tilde{\kappa}}\langle v_1, w^*\rangle-\frac{2\kappa}{\kappa+\tilde{\kappa}}\|v_2\|\\[0.3cm]
&\geq&1-\frac{2\beta}{\kappa+\tilde{\kappa}}-\frac{2\kappa}{\kappa+\tilde{\kappa}}\\[0.3cm]
&\geq&1-\frac{4\tilde{\kappa}\varepsilon+2\kappa}{\kappa+\tilde{\kappa}}\\[0.3cm]
&>&0,
\end{array}
\]
where the first inequality follows from the Cauchy-Schwarz inequality, the second  one   from the fact that $\beta\leq 2\tilde{\kappa}\varepsilon$, and the last one from our setting that $\varepsilon<  \frac{\tilde{\kappa}-\kappa}{4\tilde{\kappa}}$. Therefore, we have
\begin{equation}\label{near-kongzhi}
\displaystyle\max_{w\in \cl\pos(X-\tilde{x})\cap \Sph}\langle x^*, w\rangle>\kappa \|u^*\|.
\end{equation}
In view of (\ref{near-gph}-\ref{near-u}), (\ref{near-normal}-\ref{near-kongzhi}) and the fact that $\varepsilon$ could be any  number such that  $0<\varepsilon<  \frac{\tilde{\kappa}-\kappa}{4\tilde{\kappa}}$, we conclude that condition (\ref{key-youjie000}) cannot hold  for all  $(x,u)$ close enough to $(\bar{x}, \bar{u})$ in $\gph S|_X$, a contradiction. This completes the proof. \end{proof}

Whenever $X$ is not only closed but also convex, the gap between the previous necessary and sufficient conditions will be filled,  and even a formula for the graphical modulus ${\rm lip}_X S(\bar{x}\mid\bar{u})$ can be provided.
\begin{theorem}[Lipschitz-like property relative to a closed and convex set]\label{theo-point-convex}
 Consider $S: \R^n\rightrightarrows  \R^m$, $\bar{x}\in X\subset \R^n$ and $\bar{u}\in S(\bar{x})$. Suppose that $\gph S$ is locally closed at $(\bar{x}, \bar{u})$ and that $X$ is closed and convex. The following properties are equivalent:
 \begin{description}
   \item[(a)] $S$ has the Lipschitz-like property relative to $X$ at $\bar{x}$ for $\bar{u}$.
   \item[(b)]  There is some $\kappa\geq 0$ such that the condition
      \begin{equation}\label{key-ineq}
\|{\rm proj}_{T_X(x)}(x^*)\| \leq \kappa \|u^*\|\quad\forall x^*\in D^*S|_X(x\mid u)(u^*),
\end{equation}
holds for all $(x,u)$ close enough to $(\bar{x}, \bar{u})$ in $\gph S|_X$.
 \end{description}
 Moreover, we have
  \begin{equation}\label{first-form}
     {\rm lip}_X S(\bar{x}\mid\bar{u})=\displaystyle\limsup_{\tiny (x, u)\xrightarrow[]{\gph S|_X}(\bar{x}, \bar{u})} \sup_{u^*\in \B}\sup_{x^*\in D^*S|_X(x\mid u)(u^*)}\|{\rm proj}_{T_X(x)}(x^*)\|.
 \end{equation}
 Alternatively, the coderivative $D^*S|_X(x\mid u)$ in (\ref{key-ineq}) as well as in (\ref{first-form}) can be equivalently replaced by the regular coderivative $\widehat{D}^*S|_X(x\mid u)$ or  the proximal coderivative $D^{*{\rm prox}}S|_X(x\mid u)$.
\end{theorem}
\begin{proof}  To show the equivalent replacement, assume that the following inequality holds for all $(x,u)$ close enough to $(\bar{x}, \bar{u})$ in $\gph S|_X$:
      \begin{equation}\label{temp-key-ine}
\|{\rm proj}_{T_X(x)}(x^*)\| \leq \kappa \|u^*\|\quad\forall x^*\in D^{*{\rm prox}}S|_X(x\mid u)(u^*).
\end{equation}
Let $(x,u)$ be close enough to $(\bar{x}, \bar{u})$ in $\gph S|_X$ and let $x^*\in D^*S|_X(x\mid u)(u^*)$. By definition we have $(x^*,-u^*)\in N_{\gph S|_X}(x,u)$. By the approximation principle of normals via proximal normals \cite[Exercise 6.18]{roc}, there are some $(x_k, u_k)\rightarrow (x, u)$ with $(x_k, u_k)\in \gph S|_X$ and $(x^*_k,-u^*_k)\in N^{{\rm prox}}_{\gph S|_X}(x_k,u_k)$ such that $(x^*_k, -u^*_k)\rightarrow (x^*, -u^*)$. It then follows from  (\ref{temp-key-ine}) that for all $k$ large enough,
\[
\max\left\{\max_{w\in T_{X}(x_k)\cap \Sph}\langle x^*_k, w\rangle,\; 0\right\}=\|{\rm proj}_{T_X(x_k)}(x^*_k)\| \leq \kappa \|u^*_k\|,
\]
where the equality follows from the projection theorem \cite[Exercise 12.22]{roc} for a closed and convex cone and its polar. So we have for all $k$ large enough,
\[
\langle x^*_k, w\rangle\leq \kappa \|u^*_k\|\quad\forall w\in  T_{X}(x_k)\cap \Sph.
\]
Let $w\in  T_{X}(x)\cap \Sph$ be given arbitrarily.  Since $X$ is closed and convex (implying that $w$ is a regular tangent vector to $X$ at $x$), it follows from regular tangent cone properties \cite[Theorem 6.26]{roc} that there exists some $w_k\in T_X(x_k)$ such that $w_k\rightarrow w$. So we have for all  $k$ large enough,
 \[
\langle x^*_k, \frac{w_k}{\|w_k\|}\rangle \leq \kappa \|u^*_k\|,
\]
from which, it follows that $\langle x^*, w\rangle \leq \kappa \|u^*\|$. As $w\in  T_{X}(x)\cap \Sph$ is given arbitrarily, we have
\[
\max_{w\in T_{X}(x)\cap \Sph}\langle x^*, w\rangle \leq \kappa \|u^*\|,
\]
or equivalently
\[
\|{\rm proj}_{T_{X}(x)}(x^*)\|=\max\left\{\max_{w\in T_{X}(x)\cap \Sph}\langle x^*, w\rangle,\; 0\right\} \leq \kappa \|u^*\|,
\]
where the equality follows also from the projection theorem \cite[Exercise 12.22]{roc}.  So starting from (\ref{temp-key-ine}), we  assert that the following condition holds for all $(x,u)$ close enough to $(\bar{x}, \bar{u})$ in $\gph S|_X$:
      \[
\|{\rm proj}_{T_X(x)}(x^*)\| \leq \kappa \|u^*\|\quad\forall x^*\in D^*S|_X(x\mid u)(u^*).
\]
This, together with the inclusions
 \[
 \gph D^{*{\rm prox}}S|_X(x\mid u)\subset \gph \widehat{D}^*S|_X(x\mid u)\subset \gph D^*S|_X(x\mid u),
 \]
  indicates that the coderivative $D^*S|_X(x\mid u)$ in (\ref{key-ineq}) as well as in (\ref{first-form})  can be equivalently replaced by the regular coderivative $\widehat{D}^*S|_X(x\mid u)$ or  the proximal coderivative $D^{*{\rm prox}}S|_X(x\mid u)$ as claimed.

  In what follows, let
  \[
 \beta:=\displaystyle\limsup_{(x, u)\overset{\gph S|_X}\longrightarrow (\bar{x}, \bar{u})\;\\ } \sup_{u^*\in \B}\sup_{x^*\in D^*S|_X(x\mid u)(u^*)}\|{\rm proj}_{T_X(x)}(x^*)\|.
  \]

[(a) $\Longrightarrow$ (b)] Assuming (a), we will show (b) by proving the inequality
\begin{equation}\label{jyc}
\beta\leq {\rm lip}_X S(\bar{x}\mid\bar{u}).
\end{equation}
Choose any $\kappa\in ({\rm lip}_X S(\bar{x}\mid\bar{u}), +\infty)$. Then $S$ has the Lipschitz-like property relative to $X$ at $\bar{x}$ for $\bar{u}$ with constant $\kappa$. It then follows from Theorem \ref{theo-nece} that the following condition  holds  for all $(x,u)$ close enough to $(\bar{x}, \bar{u})$ in $\gph S|_X$:
\begin{equation}\label{key-youjie111}
\max_{w\in T_{X}(x)\cap \Sph}\langle x^*, w\rangle \leq \kappa \|u^*\|\quad \forall x^*\in \widehat{D}^*S|_X(x\mid u)(u^*).
\end{equation}
By the same argument used earlier and the equivalent replacement as we have already shown, we assert that the following condition  holds  for all $(x,u)$ close enough to $(\bar{x}, \bar{u})$ in $\gph S|_X$:
      \[
\|{\rm proj}_{T_X(x)}(x^*)\| \leq \kappa \|u^*\|\quad\forall x^*\in D^*S|_X(x\mid u)(u^*),
\]
which implies (b) and hence the inequality $\beta\leq \kappa$. So the inequality (\ref{jyc}) follows.

[(b) $\Longrightarrow$ (a)] Assuming (b), we will show (a) by proving the inequality
\begin{equation}\label{jyccc}
{\rm lip}_X S(\bar{x}\mid\bar{u})\leq \beta,
\end{equation}
from which the equality (\ref{first-form})  follows as the inequality in the other direction has been proved earlier.
  Suppose by contradiction that the inequality (\ref{jyccc}) does not hold. Choose any  $\kappa'$ such that $\beta<\kappa'< {\rm lip}_X S(\bar{x}\mid\bar{u})$.    Clearly, $S$ fails to have  the Lipschitz-like property relative to $X$ at $\bar{x}$ for $\bar{u}$ with constant $\kappa'$.  In view of the fact that  $T_X(x)=\cl\pos(X-x)$ for all $x\in X$ due to $X$ being closed and convex, we deduce from Theorem \ref{theo-suff} that the inequality
\begin{equation}\label{zjs000111}
\max_{w\in T_X(x)\cap \Sph}\langle x^*, w\rangle \leq \kappa' \|u^*\|
\end{equation}
cannot be fulfilled  for all  $(x,u)$ close enough to $(\bar{x}, \bar{u})$ in $\gph S|_X$ and  $(x^*,-u^*)\in N_{\gph S|_X}(x,u)$. By the same argument used earlier, the inequality (\ref{zjs000111}) amounts to
\[
\|{\rm proj}_{T_{X}(x)}(x^*)\|\leq \kappa' \|u^*\|.
\]
 So there exist  some
$(x_k,u_k)\to(\bar{x},\bar{u})$ with $(x_k,u_k)\in \gph S|_X$ and some $(x_k^*,-u_k^*)\in N_{\gph S|_X}(x_k, u_k)$ such that
 $\|{\rm proj}_{T_{X}(x_k)}(x^*_k)\|> \kappa' \|u^*_k\|$   for all $k$. Let $w_k:={\rm proj}_{T_{X}(x_k)}(x^*_k)$ for all $k$.  Clearly, we have $\|w_k\|>0$ for all $k$. Then we have for all $k$,
\[
\kappa'\frac{\|u_k^*\|}{\|w_k\|}<1,
\]
 and
\[
{\rm proj}_{T_{X}(x_k)}\left(\kappa'\frac{x^*_k}{\|w_k\|}\right)=\kappa'\frac{w_k}{\|w_k\|}.
\]
Since $\kappa'(\frac{x^*_k}{\|w_k\|},-\frac{u_k^*}{\|w_k\|})\in N_{\gph S|_X}(x_k, u_k)$ or equivalently $\kappa'\frac{x^*_k}{\|w_k\|}\in D^*S|_X(x_k\mid u_k)(\kappa'\frac{u_k^*}{\|w_k\|})$ for all $k$, we have
\[
\beta\geq \displaystyle\limsup_{k \to +\infty\;\\ } \left\|{\rm proj}_{T_{X}(x_k)}\left(\kappa'\frac{x^*_k}{\|w_k\|}\right)\right\|=\kappa',
\]
contradicting to the assumption that $\beta<\kappa'$.    This completes the proof. \end{proof}

Motivated from Theorem \ref{theo-point-convex}, we can provide a point-based criterion for the relative  Lipschitz-like property via the projectional coderivative defined below by first projecting the coderivative for all nearby points onto the tangent cones and then taking the outer limits for the projections.
\begin{definition}[Projectional coderivatives]\label{def-proj-outer-norm}
Consider a mapping $S: \R^n\rightrightarrows \R^m$ and  a point $\bar{x}\in X\subset \R^n$.
The  projectional  coderivative of $S$ at $\bar{x}$ for any $\bar{u}\in S(\bar{x})$ with respect to $X$ is the mapping $D^*_{X}S(\bar{x}\mid \bar{u}):\R^m\rightrightarrows \R^n$ defined by
\[
x^*\in D^*_{X}S(\bar{x}\mid \bar{u})(u^*)\Longleftrightarrow (x^*, -u^*)\in \limsup_{\tiny (x,u)\xrightarrow[]{\gph S|_X}(\bar{x}, \bar{u})}
{\rm proj}_{T_X(x)\times \R^m}N_{\gph S|_X}(x,u).
\]
That is, $x^*\in D^*_{X}S(\bar{x}\mid \bar{u})(u^*)$ if and only if there are some $(x_k,u_k){\tiny \xrightarrow[]{\gph S|_X}}(\bar{x},\bar{u})$ and
$x_k^*\in D^*S|_X(x_k\mid u_k)(u_k^*)$ such that $u_k^*\to u^*$ and ${\rm proj}_{T_X(x_k)}(x_k^*)\to x^*$. Here the notation  $D^*_{X}S(\bar{x}\mid \bar{u})$   is simplified to   $D^*_{X}S(\bar{x})$ when $S$ is single-valued at $\bar{x}$, i.e., $S(\bar{x}) =\{\bar{u}\}$.
\end{definition}

\begin{theorem}[generalized Mordukhovich criterion]\label{theo-convex-aubin}
 Consider $S: \R^n\rightrightarrows  \R^m$, $\bar{x}\in X\subset \R^n$ and $\bar{u}\in S(\bar{x})$. Suppose that $\gph S$ is locally closed at $(\bar{x}, \bar{u})$ and that $X$ is closed and convex.  The following properties are equivalent:
 \begin{description}
   \item[(a)] $S$ has the Lipschitz-like property relative to $X$ at $\bar{x}$ for $\bar{u}$.
   \item[(b)] $D^*_{X}S(\bar{x}\mid\bar{u})(0)=\{0\}$.
   \item[(c)] $|D^*_X S(\bar{x}\mid\bar{u})|^+<+\infty$.
 \end{description}
Furthermore, we have
\begin{equation}\label{wanmei}
{\rm lip}_X S(\bar{x}\mid\bar{u})=|D^*_X S(\bar{x}\mid\bar{u})|^+.
\end{equation}
\end{theorem}
\begin{proof} It is clear to see from  the definition of projectional coderivatives  that,  the mapping $D^*_{X}S(\bar{x}\mid\bar{u})$ is outer semicontinuous  and positively homogeneous. Then the equivalence of (b) and (c)  follows immediately from \cite[Proposition 9.23]{roc}. The equivalence of (a) and (c), and the formula for ${\rm lip}_X S(\bar{x}\mid\bar{u})$ can be proved in a similar way as  in the proof of Theorem \ref{theo-point-convex}. The detailed proof is omitted. This completes the proof. \end{proof}

The important role played by the projectional coderivative $D^*_{X}S(\bar{x}\mid\bar{u})$ in the study of the relative Lipschitz-like property,  is revealed by the generalized Mordukhovich criterion above.  In the following remarks, we list some simple facts on the projectional coderivatives.

\begin{remark}\label{rem-classical}
In the case of $\bar{x}\in \inte X $, we have for all $\bar{u}\in S(\bar{x})$,
\[
D^*_{X}S(\bar{x}\mid\bar{u})=D^*S(\bar{x}\mid\bar{u}),
\]
the results in Theorem  \ref{theo-convex-aubin} as well as in Theorem \ref{theo-point-convex}  recover the  Mordukhovich criterion for the `classical' Lipschitz-like property with no restriction on any set, see \cite[Theorem 9.40]{roc}.
\end{remark}

\begin{remark}[projectional coderivatives of smooth mappings with respect to sets with simple structures]\label{ex-smooth-hyperplane}
Consider  a smooth, single-valued mapping $F:\R^n\rightarrow \R^m$. By some coderivative calculus in \cite[Example 8.34 and Exercise 10.43]{roc}, we can obtain some formulas for the projectional coderivatives of $F$ with respect to sets having simple structures.  In the case of an affine set
 \[
 X:=\{x\in \R^n\mid Bx=b\},
 \]
 where $B$ is an $m\times n$ matrix and $b\in \R^m$, we have for all $\bar{x}\in \bdry X$,
\[
D^*_XF(\bar{x})(y)={\rm proj}_{{\rm ker}\, B}(\nabla F(\bar{x})^*y),
\]
where ${\rm ker}\, B:=\{x\in \R^n\mid Bx=0\}$.  While in the case of   a closed half-space \[X:=\{x\mid \langle a, x\rangle\leq \beta\},\]
we have for all $\bar{x}\in \bdry X$,
\[
D^*_XF(\bar{x})(y)=\left\{
\begin{array}{ll}
\left[\nabla F(\bar{x})^*y,\;\; {\rm proj}_{[a]^\perp}(\nabla F(\bar{x})^*y)\right] & \mbox{if}\;\langle \nabla F(\bar{x})^*y, a\rangle \leq 0,\\[0.2cm]
\left\{\nabla F(\bar{x})^*y,\;\; {\rm proj}_{[a]^\perp}(\nabla F(\bar{x})^*y)\right\} & \mbox{if}\;\langle \nabla F(\bar{x})^*y, a\rangle > 0.
\end{array}
\right.
\]
\end{remark}

\begin{remark}[projectional coderivative  of  the solution mapping of a linear system]\label{rem-3derivatives}
Consider the solution mapping
\begin{equation}\label{defofLS}
  S: p\mapsto \left\{x\in \R^n \mid  Ax + p \in K\right\}
\end{equation}
of a linear system, where $A\in \R^{m\times n}$ is a matrix, $p\in \R^m$ is some parameter and  $K\subset \R^m$ is a convex polyhedron.
 Clearly, we have
$\dom S=K+\rg A$, which is a convex polyhedron but not necessarily the whole space $\R^m$.    By \cite[ Exercises  6.7 and 6.44]{roc},   we have for any  $(p, x)\in \gph S$ or equivalently $Ax+p\in K$,
\[
   N_{\dom S}(p)=N_K(Ax + p) \cap \ker A^T\;\mbox{and}\; N_{\gph S}(p, x)=\left\{ (y, A^T y )\mid y\in N_K(Ax + p )\right\}.
\]
Let $(\bar{p}, \bar{x})\in \gph S$ and let $\mathcal{F}(\bar{p},\bar{x})$ be the  collection  of  the faces of $K$ that contain $A\bar{x}+\bar{p}$. Then we have \begin{equation}\label{tysds-gs}
\gph D^*_{\dom S}S(\bar{p}\mid \bar{x})\;=\bigcup_{F\in \mathcal{F}(\bar{p},\bar{x})}\left\{\left(-A^Ty,\;\;y-{\rm proj}_{N^F\cap \ker A^T}(y)\right)\mid y\in N^F\right\},
\end{equation}
where  $N^F:=N_K(Ax+p)$  for any $ F\in \mathcal{F}(\bar{p},\bar{x})$ and any choice of $(p, x)$ such that $Ax+p\in \ri F$.  To show (\ref{tysds-gs}), we rely on  the definition of the projectional coderivative by combining the following facts: (i)
$\mathcal{F}(\bar{p},\bar{x})$ consists of finitely many faces of $K$; (ii) for any  face $F$ of $K$, $T_K(Ax+b)$ and $N_K(Ax+b)$ are both constants whenever $Ax+p\in \ri F$;  (iii) for any sequence $(p_k, x_k)\to (\bar{p}, \bar{x})$ with $(p_k, x_k)\in \gph S$  for all $k$, there exists  some $F\in \mathcal{F}(\bar{p},\bar{x})$ such that, by taking a subsequence if necessary, $Ax_k+p_k\in \ri F$ for all $k$;   (iv) for any $F\in \mathcal{F}(\bar{p},\bar{x})$ and any $(p, x)$ with $Ax+p\in \ri F$,
\[
{\rm proj}_{T_{\dom S}(p)}(y)=y-{\rm proj}_{N_{\dom S}(p)}(y)=y-{\rm proj}_{N^F\cap \ker A^T}(y)\quad \forall y;
\]
and (v)  for any $F\in \mathcal{F}(\bar{p},\bar{x})$,
\[
\limsup_{\tiny (p, x)\xrightarrow[]{Ax+p\in \ri F}(\bar{p}, \bar{x})}
{\rm proj}_{T_{\dom S}(p)\times \R^n}N_{\gph S}(p,x)=\left\{\left(y-{\rm proj}_{N^F\cap \ker A^T}(y),\;\;A^Ty\right)\mid y\in N^F\right\}.
\]
In contrast,  by the definition of the coderivative,  we have
\[
\gph D^*S(\bar{p}\mid \bar{x})= \left\{\left(-A^Ty,\;\;y \right)\mid y\in N_K(A\bar{x}+\bar{p})\right\}.
\]
The classical Mordukhovich criterion $D^*S(\bar{p}\mid \bar{x})(0)=\{0\}$ amounts to
\[
N_K(A\bar{x}+\bar{p})\cap \ker A^T=\{0\}\Longleftrightarrow T_K(A\bar{x}+\bar{p})+\rg A=\R^m \Longleftrightarrow   \bar{p}\in \inte(\dom S).
\]
While the generalized Mordukhovich criterion $D^*_{\dom S}S(\bar{p}\mid \bar{x})(0)=\{0\}$ holds automatically  as it amounts to the following trivial equalities:
\[
y-{\rm proj}_{N^F\cap \ker A^T}(y)\equiv 0\quad \forall F\in \mathcal{F}(\bar{p},\bar{x}),\; \forall y\in N^F\cap \ker A^T.
\]
It is interesting to note that
\[
\gph D^*_{\dom S}S(\bar{p}\mid \bar{x})=\gph D^*S(\bar{p}\mid \bar{x})  \Longleftrightarrow \bar{p}\in \inte(\dom S),
\]
meaning that the projectional coderivative $D^*_{\dom S}S(\bar{p}\mid \bar{x})$ and the coderivative $D^*S(\bar{p}\mid \bar{x})$ differs from each other only when  $\bar{p}$  is on the boundary of   $\dom S$.
\end{remark}

  As suggested by one reviewer, we will compare our sufficient condition with the one established for an implicitly defined set-valued mapping in terms of a directional limiting coderivative in \cite[Theorem 3.5]{bgo19}.
To have a better comparison, we first present an explicit version of \cite[Theorem 3.5]{bgo19} as follows.

We start by recalling the definitions  of the directional limiting normal cone and the directional limiting coderivative. For a set  $\Omega\subset \R^n$ with $\bar{x}\in \Omega$  and a direction $u\in \R^n$, the directional limiting normal cone to $\Omega$ in direction $u$ at $\bar{x}$ is defined by
\[
N_\Omega(\bar{x};u):=\limsup_{t\downarrow 0, \;u'\to u}\widehat{N}_\Omega(\bar{x}+tu'),
\]
while for a set-valued mapping $S:\R^n\rightrightarrows  \R^m$ having locally closed graph around $(\bar{x}, \bar{y})\in \gph S$ and a pair of directions $(u, v)\in \R^n\times \R^m$, the set-valued mapping $D^*S((\bar{x},\bar{y});(u, v)):\R^m\rightrightarrows \R^n$, defined by
\[
D^*S((\bar{x},\bar{y});(u, v))(v^*):=\{u^*\in \R^n\mid (u^*, -v^*)\in N_{\gph S}((\bar{x},\bar{y});(u, v))\}\quad \forall v^*\in \R^m
\]
is called the directional limiting coderivative of $S$ in the direction $(u, v)$ at $(\bar{x}, \bar{y})$. See \cite{bgo19,gfr14} and references therein  for more details and some basic properties of these notions.

\begin{theorem}[{\cite[Theorem 3.5]{bgo19}} in an explicit form]\label{theo-fxsds}  Consider $S: \R^n\rightrightarrows  \R^m$, $\bar{x}\in X\subset \R^n$ and $\bar{u}\in S(\bar{x})$. Assume  that $\gph S$ is locally closed at $(\bar{x}, \bar{u})$ and that $X$ is closed. Further assume that the following conditions are satisfied:
\begin{description}
  \item[(i)]  For every $x \in T_X (\bar{x})$ and every sequence $t_k\downarrow 0$, there exists some $u\in \R^n$ such that
\[
\liminf_{k\to \infty}\frac{d((\bar{x}+t_kx,\bar{u}+t_ku), \gph S)}{t_k}=0.
\]
This holds in particular if, for every $x \in T_X (\bar{x})$, there is some $u\in \R^m$ such that $(x,u)\in T_{\gph S}(\bar{x}, \bar{u})$ is derivable (i.e., for every $t_k\downarrow 0$, there is some $(x_k, u_k)\to (x, u)$ such that $(\bar{x}, \bar{u})+t_k(x_k, u_k)\in \gph S$ for all $k$).
  \item[(ii)] The  equality
  \[
D^*S \left( (\bar{x},\bar{u}) ;(x, u) \right)(0) = \{0\}
\]
holds for all $x \in T_X(\bar{x})$ and $(x,u)\in T_{\gph S}(\bar{x}, \bar{u})$ with $(x, u)\not=(0, 0)$.
  \end{description}
Then $S$ has the Lipschitz-like property relative to $X$ at $\bar{x}$ for $\bar{u}$. 
\end{theorem}
\begin{proof} Clearly, we have $S(x)=\{u\in \R^m\mid 0\in M(x, u)\}$ for all $x\in \R^n$, where $M(x,u):=S(x)-u$. It is clear to see that  $\gph S$ is locally closed at $(\bar{x}, \bar{u})$ if and only if  $\gph M$ is locally closed at $(\bar{x}, \bar{u}, 0)$, and that condition (i) holds if and only if, for every $x \in T_X (\bar{x})$ and every sequence $t_k\downarrow 0$, there exists some $u\in \R^m$ such that
\[
\liminf_{k\to \infty}\frac{d((\bar{x}+t_kx,\bar{u}+t_ku, 0), \gph M)}{t_k}=0.
\]
Moreover, by definition, we have the following equivalences:
\[
(x,u,0)\in T_{\gph M}(\bar{x}, \bar{u}, 0) \Longleftrightarrow (x,u)\in T_{\gph S}(\bar{x}, \bar{u}),
\]
and
\[
(x^*, 0)\in D^*M((\bar{x}, \bar{u}, 0);(x, u, 0))(y^*)\Longleftrightarrow y^*=0\;\mbox{and}\; x^*\in D^*S((\bar{x}, \bar{u});(x, u))(0),
\]
from which, it follows that   condition (ii) holds if and only if, for every nonzero $(x, u) \in T_X(\bar{x})\times \R^m$ with $(x,u,0)\in T_{\gph M}(\bar{x}, \bar{u}, 0)$,   $(x^*, 0)$ belongs to  $D^*M((\bar{x}, \bar{u}, 0);(x, u, 0))(y^*)$ only if $x^*=0$ and $y^*=0$. Therefore, \cite[Theorem 3.5]{bgo19} can be applied in a direct way to obtain the result. \end{proof}

\begin{remark}\label{rem-xxbj} In the case of $X$ being merely closed,  two sufficient conditions are provided, respectively,  in Theorems \ref{theo-suff} and \ref{theo-fxsds} for the Lipschitz-like property of  $S$ relative to $X$ at $\bar{x}$ for $\bar{u}$.    Unlike Theorem \ref{theo-suff}, which utilizes  integrated information behind the coderivative of the restriction  mapping $S|_X$ (combining the local behavior of $S$ around $(\bar{x}, \bar{u})$ and  of $X$ around $\bar{x}$ as a whole),  Theorem \ref{theo-fxsds} treats $S$ and $X$ as `separated variables' (the local behavior of $S$ around $(\bar{x}, \bar{u})$ described by the directional limiting coderivative $D^*S \left( (\bar{x},\bar{u}) ;(\cdot, \cdot) \right)$ is independent of the local behavior of  $X$ around $\bar{x}$).  So it would be  the case that the sufficient condition in Theorem \ref{theo-fxsds} is easier to be verified than that in  Theorem \ref{theo-suff}. However, in the case of $X$ being not only closed but also convex,  the sufficient condition in Theorem \ref{theo-suff} turns out to be also necessary as can be seen from Theorems \ref{theo-point-convex} and \ref{theo-convex-aubin},   but the sufficient condition  in Theorem \ref{theo-fxsds} is  far from being necessary as will be seen from Example \ref{exam-lcp} below.
\end{remark}

To end this section, we demonstrate by an interesting example how our results in Theorems \ref{theo-point-convex} and \ref{theo-convex-aubin} can be applied in the circumstance  that the graph  can be decomposed into finitely many simple pieces. Moreover, by this example,  we also demonstrate how Theorem \ref{theo-fxsds} can   fail in identifying the relative  Lipschitz-like property.
\begin{example}\label{exam-lcp}
Consider the solution mapping $S:\R^2\rightrightarrows \R^2$ of  a linear complementarity system:
\begin{equation} \label{LCPexample}
S(q) := \left\{ x\in \R^2\mid x \geq 0,\; Mx + q \geq 0,\; \langle x, \, Mx+q \rangle =0 \right\},
\end{equation}
where
\[
 M = \left[\begin{matrix}
                -1 & 0 \\
                1 & 1
              \end{matrix}
              \right].
\]
Clearly, we have
\[
\dom S=\R_+\times \R\quad\mbox{and}\quad\gph S=\left\{ (q, x)\in \R^2\times \R^2\mid x \geq 0,\; Mx + q \geq 0,\; \langle x, \, Mx+q \rangle =0 \right\}.
\]
In terms of
\[
\mathcal{I}:=\{(I_1, I_2, I_3)\mid  I_1\cup I_2\cup I_3=\{1, 2\},\;I_i\cap I_j=\emptyset\;\forall i\not=j\}
\]
and
\[
(\gph S)_{(I_1, I_2, I_3)}:=\left\{
(q, x)\in \R^2\times \R^2\left|  \begin{array}{lll}
  x_i=0 & (Mx+q)_i>0 &\mbox{if}\; i\in I_1 \\[0.2cm]
  x_i>0 & (Mx+q)_i=0 &\mbox{if}\; i\in I_2\\[0.2cm]
  x_i=0 & (Mx+q)_i=0 &\mbox{if}\; i\in I_3
\end{array}
\right.
\right\}\; \forall (I_1, I_2, I_3)\in \mathcal{I},
\]
we have
\[
\gph S=\displaystyle\bigcup_{(I_1, I_2, I_3)\in \mathcal{I}}(\gph S)_{(I_1, I_2, I_3)}.
\]
Note that $(\gph S)_{(I_1, I_2, I_3)}\not=\emptyset$ for all $(I_1, I_2, I_3)\in \mathcal{I}$ and that
\[
(\gph S)_{(I_1, I_2, I_3)}\cap (\gph S)_{(I'_1, I'_2, I'_3)}=\emptyset\; \forall (I_1, I_2, I_3)\not=(I'_1, I'_2, I'_3).
\]
Then for every  $(q, x)\in \gph S$, there is a unique $(I_1, I_2, I_3)\in \mathcal{I}$ such that $(q, x)\in (\gph S)_{(I_1, I_2, I_3)}$ and
\begin{equation}\label{ydys}
N_{\gph S}(q, x)=\left\{(u^*,\;M^Tu^*+v^*)\left|
\begin{array}{ll}
  (u^*_i,v^*_i)\in \{0\}\times \R &\mbox{if}\; i\in I_1 \\[0.2cm]
(u^*_i,v^*_i)\in \R\times \{0\}&\mbox{if}\; i\in I_2\\[0.2cm]
  (u^*_i,v^*_i)\in  \Omega&\mbox{if}\; i\in I_3
\end{array}
\right.
\right\},
\end{equation}
where $\Omega:=(\R\times \{0\}) \cup (\{0\} \times \R) \cup \R^2_-$.

Consider in particular $(\bar{q}, \bar{x})=(0, 0)\in \gph S$. We have $(\bar{q}, \bar{x})\in (\gph S)_{(\emptyset, \emptyset, \{1, 2\})}$  and hence
\[
N_{\gph S}(\bar{q}, \bar{x})=\left\{(u^*,\;M^Tu^*+v^*)\left|
\begin{array}{ll}
   (u^*_i,v^*_i)\in \Omega&\forall i=1,2
\end{array}
\right.
\right\}.
\]
This implies by definition that
\[
D^*S(\bar{q}\mid \bar{x})(0)=\R_-\times \{0\}\not=\{(0, 0)\}.
\]
So by the  Mordukhovich criterion   \cite[Theorem 9.40]{roc}, we assert that $S$ does not have the   Lipschitz-like property at $\bar{q}$ for $\bar{x}$, which can also be seen from the fact that $\bar{q}\in \bdry \dom S$.

In what follows, we will apply Theorem \ref{theo-point-convex} to study the Lipschitz-like property of $S$ relative to $\dom S$ at $\bar{q}$ for $\bar{x}$, which amounts to the existence of some $\kappa \geq 0$ such that
      \begin{equation}\label{key-ineq1111}
\|{\rm proj}_{T_{\dom S}(q)}(u^*)\| \leq \kappa \|M^Tu^*+v^*\|\quad\forall  (u^*,\;M^Tu^*+v^*)\in N_{\gph S}(q, x)
\end{equation}
holds for all $(q, x)$ close to $(\bar{q}, \bar{x})$ in $\gph S$, or in other words, for all $(q, x)$ close to  $(\bar{q}, \bar{x})$ in all
$(\gph S)_{(I_1, I_2, I_3)}$    with $(I_1, I_2, I_3)\in \mathcal{I}$. For each $(I_1, I_2, I_3)\in \mathcal{I}$, we define
\[
\kappa(I_1, I_2, I_3):=\inf\left\{\kappa\geq 0 \mid \mbox{(\ref{key-ineq1111}) holds for all}\; (q, x)\in (\gph S)_{(I_1, I_2, I_3)} \right\}.
\]
Then by some direct calculation, we have
\begin{equation}\label{kap}
\kappa(I_1, I_2, I_3)=\left\{
\begin{array}{ll}
  0 & \mbox{if}\; (I_1, I_2, I_3)=(\{1, 2\}, \emptyset, \emptyset)\\[0.2cm]
1 &\mbox{if}\; (I_1, I_2, I_3)\in \{(\{2\},\emptyset, \{1\}), (\{1\},\{2\}, \emptyset), (\{1\}, \emptyset, \{2\}), (\{2\},\{1\}, \emptyset)\} \\[0.2cm]
        \sqrt{\frac{3+\sqrt{5}}{2}} & \mbox{otherwise}.
\end{array}
\right.
\end{equation}
For instance, whenever  $(I_1, I_2, I_3)=(\{2\},\emptyset, \{1\})$, we have for all $(q, x)\in (\gph S)_{(I_1, I_2, I_3)}=\{(q, x)\mid q_1=0,\; q_2>0,\; x_1=x_2=0\}$,
 $$N_{\gph S} (q, x) = \left\{\left( \left(\begin{matrix}
                                            u^*_1 \\
                                            0
                                          \end{matrix} \right),
                                          \left( \begin{matrix}
                                            -u^*_1 +v^*_1\\
                                             v^*_2
                                          \end{matrix} \right)  \right) \ \Bigg|  \ v^*_2 \in \R ,(u^*_1, v^*_1 ) \in \Omega \right\},$$
     $${\rm proj}_{T_{\dom S}(q)\times \R^2}N_{\gph S} (q, x) = \left\{\left( \left(\begin{matrix}
                                            \max\{u^*_1, 0\} \\
                                            0
                                          \end{matrix} \right),
                                          \left( \begin{matrix}
                                            -u^*_1 +v^*_1\\
                                             v^*_2
                                          \end{matrix} \right)  \right) \ \Bigg|  \ v^*_2 \in \R ,(u^*_1, v^*_1 ) \in \Omega \right\},$$
and
\[
\kappa(I_1, I_2, I_3)=1.
\]
For another instance, whenever  $(I_1, I_2, I_3)=(\emptyset, \{1, 2\},\emptyset)$, we have for all $(q, x)\in (\gph S)_{(I_1, I_2, I_3)}=\{(q, x)\mid q_1>0,\; q_2<-q_1,\; x_1=q_1,\; x_2 = -q_1 - q_2\}$,
 $${\rm proj}_{T_{\dom S}(q)\times \R^2}N_{\gph S} (q, x) = N_{\gph S} (q, x) = \left\{\left(  u^*, M^Tu^*   \right) \mid  u^*\in \R^2 \right\},$$
and
\[
\kappa(I_1, I_2, I_3)=\frac{1}{\min_{y\in \Sph } \|M^T y\| }= \sqrt{\frac{3+\sqrt{5}}{2}}.
\]
In view of (\ref{kap}), we get from Theorem \ref{theo-point-convex} that
  \[
     {\rm lip}_{\dom S} S(\bar{q}\mid\bar{x})=\max\{\kappa(I_1, I_2, I_3)\mid (I_1, I_2, I_3)\in \mathcal{I}\}= \sqrt{\frac{3+\sqrt{5}}{2}},
 \]
 and hence that $S$ does have the Lipschitz-like property  relative to $\dom S$ at $\bar{q}$ for $\bar{x}$. In contrast,  we can also apply Theorem \ref{theo-convex-aubin} to obtain the same result by calculating the projectional coderivative $D^*_{\dom S}S(\bar{q}\mid \bar{x})$, which can be done in the same way by decomposing $\gph S$ into finitely many pieces $(\gph S)_{(I_1, I_2, I_3)}$. The details are complicated and thus omitted.

Let $Q$ be a closed subset of $\dom S$ such that $\bar{q}\in Q$ and $q:=(0, -1)^T\in T_Q(\bar{q})$ (e.g.,  $\dom S$ could be the largest instance of $Q$). Our argument above suggests that $S$ has the Lipschitz-like property  relative to $Q$ at $\bar{q}$ for $\bar{x}$, which, however, cannot be verified via Theorem \ref{theo-fxsds}. To this end, we argue that  the condition $D^*S((\bar{q},\bar{x}); (q,x))(0)=\{0\}$ cannot be fulfilled when $x:=(0, 1)^T$.  By definition, we have in terms of $(I_1, I_2, I_3):=(\emptyset, \{2\}, \{1\})$,
\[
(q, x)\in (\gph S)_{(I_1, I_2, I_3)}\subset \gph S.
\]
 In view of the facts that $\gph S$ is the union of finitely many convex polyhedral cones and that $(\bar{q}, \bar{x})=(0, 0)$, we have $(q, x)\in T_{\gph S}(\bar{q}, \bar{x})=\gph S$  and
\[
\begin{array}{ll}
 N_{\gph S} \left( (\bar{q},\bar{x}); (q,x) \right)&:=\limsup_{t\downarrow 0, \ (q',x') \rightarrow (q,x)} N_{\gph S} \left(   t (q',x') \right)\\[0.5cm]
&=\limsup_{  (q',x') \rightarrow (q,x)} N_{\gph S} \left(    q',x'  \right)\\[0.5cm]
&=  N_{\gph S} \left(    q,x  \right)\\[0.5cm]
&= \left\{(u^*,\;M^T u^*+v^*) \left|
  (u^*_1,v^*_1)\in \Omega,\;
(u^*_2,v^*_2)\in \R\times \{0\}
\right.
\right\},
\end{array}
\]
where the    second equality follows from the fact that $\gph S$ is a closed cone, the third one   from the outer semi-continuity  of the normal cone mappings, and the last one  from (\ref{ydys}). Then  we have  by definition,
\[
D^*S((\bar{q},\bar{x}); (q,x) ) (0) = \R_-\times \{0\}\not=\{(0, 0)^T\},
\]
suggesting that the Lipschitz-like property of $S$  relative to $Q$ at $\bar{q}$ for $\bar{x}$ cannot be derived from Theorem \ref{theo-fxsds}.
\end{example}

\section{Profile mappings and relative Lipschitzian continuity}\label{sec-prof-map}
Consider a function $f:\R^n\rightarrow \overline{\R}$, a point $\bar{x}\in \R^n$  where $f$ is finite and locally lsc, and a set $X\subset \dom f$ such that  $\bar{x}\in X$.
The notion of  relative  Lipschitzian continuity of $f$ is standard, see \cite[Definition 9.1 (b)]{roc} for a formal definition.  To say that  $f$ is locally Lipschitz continuous at $\bar{x}$ relative to $X$ is to assert the following inequality:
\[
{\lip}_Xf(\bar{x}):=\displaystyle\limsup_{\tiny\begin{array}{c}
  x,x'\xrightarrow[]{X} \bar{x} \\
 x\not=x'
\end{array}}\frac{|f(x)-f(x')|}{\|x-x'\|}<+\infty.
\]
It is straightforward  to verify that  $f$ is locally Lipschitz continuous at $\bar{x}$ relative to $X$ if and only if the profile mapping
 \[
 E_f: x\mapsto \{\alpha\in \R\mid \alpha\geq f(x)\}
 \]
  has the Lipschitz-like property relative to $X$ at $\bar{x}$ for $f(\bar{x})$, and furthermore that their moduli are equal:
   \begin{equation}\label{modulus-equal}
{\lip}_Xf(\bar{x})={\lip}_X E_f(\bar{x}\mid f(\bar{x})).
 \end{equation}

  In what follows, we will study the relative Lipschitzian continuity by applying Theorem \ref{theo-point-convex} to the profile mapping $E_f$, and will  give subgradient characterizations for the relative Lipschitzian continuity.

  To begin, we present a useful property of the proximal normals to epigraphs.
\begin{lemma}[Proximal normals to epigraphs]\label{lem-pro-epi}
For a function $f:\R^n\rightarrow \overline{\R}$ and a point $(x, \alpha)$ with $\alpha>f(x)>-\infty$, we have
\[
N^{\rm prox}_{\epi f}(x, \alpha)\subset N^{\rm prox}_{\epi f}(x, f(x))\cap (\R^n\times\{0\}).
\]
\end{lemma}
\begin{proof} Let $(v, \lambda)\in N^{\rm prox}_{\epi f}(x, \alpha)$ with $\|(v,\lambda)\|=1$. By definition, there exists  some $\delta>0$   such that
\begin{equation}\label{pro-def}
\B_\delta((x, \alpha)+\delta(v,\lambda))\cap \epi f=\{(x, \alpha)\},
\end{equation}
implying that
   \[
   (x, f(x))\in \B_\delta((x, f(x))+\delta(v,\lambda))\cap \epi f.
    \]
    Let  $(\tilde{x}, \tilde{\alpha})\in \B_\delta((x, f(x))+\delta(v,\lambda))\cap \epi f$.  Then we have
\[
(\tilde{x}, \tilde{\alpha}+\alpha-f(x))\in \B_\delta((x, \alpha)+\delta(v,\lambda))\cap \epi f,
\]
and hence by (\ref{pro-def}),
\[
(\tilde{x}, \tilde{\alpha}+\alpha-f(x))=(x, \alpha)\Longleftrightarrow (\tilde{x}, \tilde{\alpha})=(x, f(x)).
\]
The latter equation implies that
\[
 \B_\delta((x, f(x))+\delta(v,\lambda))\cap \epi f=\{(x, f(x))\}.
\]
By definition, we have $(v, \lambda)\in N^{\rm prox}_{\epi f}(x, f(x))$, implying that  $\lambda\leq 0$. We claim that $\lambda=0$, for otherwise there is some $\varepsilon>0$ such that $(x, \alpha+\varepsilon\delta\lambda)\in \B_\delta((x, \alpha)+\delta(v,\lambda))\cap \epi f$ but $(x, \alpha)\not=(x, \alpha+\varepsilon\delta\lambda)$, a contradiction to (\ref{pro-def}).
This completes the proof.
\end{proof}

The classical local Lipschitzian continuity of $f$ at $\bar{x}$ (without  mention  of $X$) has been fully characterized by virtue of the (horizon) subgradients of $f$ at $\bar{x}$ in \cite[Theorem 9.13]{roc}, which says that $f$ is locally Lipschitz continuous at $\bar{x}$ if and only if $\partial ^\infty f(\bar{x})=\{0\}$ or equivalently $\partial f(x)$ is locally bounded at $\bar{x}$, and in that case,
\[
{\lip} f(\bar{x})=\max_{v\in \partial f(\bar{x})}\|v\|.
 \]
 In parallel fashion  we can characterize the local  Lipschitzian continuity of $f$ at $\bar{x}$ relative to $X$ by utilizing the following notion of projectional (horizon) subgradients, whose construction are motivated by applying Theorem \ref{theo-point-convex} to the profile mapping $E_f$.

\begin{definition}[projectional subgradients]\label{def-proj_subg} Consider a function $f:\R^n\rightarrow\overline{\R}$, a point $\bar{x}$ with $f(\bar{x})$ finite, and a convex set $X$ with $\bar{x}\in  X$. For a vector $v\in \R^n$, we say that
 \begin{description}
   \item[(a)] $v$ is a   projectional subgradient of $f$ at $\bar{x}$ with respect to $X$, written $v\in \partial_X f(\bar{x})$, if there are sequences $\tiny x_k\xrightarrow[]{\;f+\delta_X\;} \bar{x}$  and $v_k\in \partial (f+\delta_X)(x_k)$ with ${\rm proj}_{T_X(x_k)}(v_k)\to v$;
   \item[(b)]  $v$ is a  horizon  projectional subgradient of $f$ at $\bar{x}$ with respect to $X$, written $v\in \partial^\infty_X f(\bar{x})$, if there are sequences $\lambda_k\downarrow  0$,  $\tiny x_k\xrightarrow[]{\;f+\delta_X\;} \bar{x}$  and
$v_k\in \partial (f+\delta_X)(x_k)$ with $\lambda_k{\rm proj}_{T_X(x_k)}(v_k)\to v$.
  \end{description}
Here, $\tiny x_k\xrightarrow[]{\;f+\delta_X\;} \bar{x}$ amounts to  $x_k\to \bar{x}$ with $f(x_k)\to f(\bar{x})$ and $x_k\in X$ for all $k$. In the case that  $x\not\in
  X$, we define $\partial_X f(x):=\emptyset$ and $\partial^\infty_X f(x):=\emptyset$.
\end{definition}

By using the notion of  projectional (horizon) subgradients, we can extend \cite[Theorem 9.13]{roc} to
 deal with the   Lipschitzian continuity  of a function relative to some closed and convex set, and also the Lipschitz modulus.

\begin{theorem}\label{the-locally-lip-X}
Consider a function $f:\R^n\rightarrow\overline{\R}$, a point $\bar{x}\in \R^n$  where $f$ is finite and locally lsc, and a closed and convex set  $X\subset \dom f$ such that  $\bar{x}\in X$.  Then  the following conditions  are equivalent:
\begin{description}
  \item[(a)]  $f$ is locally Lipschitz continuous at $\bar{x}$ relative to $X$.
  \item[(b)] The mapping $x\mapsto {\rm proj}_{T_X(x)}  \partial (f+\delta_X)(x) $ is locally bounded at $\bar{x}$.
  \item[(c)] The mapping $x\mapsto   \partial_X f(x) $ is locally bounded at $\bar{x}$.
        \item[(d)]  $\partial_X^\infty f(\bar{x})=\{0\}$.
 \end{description}
Moreover, when these conditions hold,   the following properties hold:
\begin{description}
         \item[(i)] The inclusion
       \begin{equation}\label{kongzhi-inclusion}
       \partial^\infty (f+\delta_X)(x)\subset N_X(x)
       \end{equation}
        holds for all $x$ close enough to $ \bar{x} $ in $X$.
        \item[(ii)] The projectional coderivative of  $E_f$ at $\bar{x}$ for $f(\bar{x})$ with respect to $X$ is given by
         \begin{equation}\label{proj-prof-co-der}
D^*_X E_f\left(\bar{x}\mid f(\bar{x})\right)(\lambda)=\left\{
\begin{array}{ll}
\lambda\partial_Xf(\bar{x}) &\mbox{if}\;\lambda>0,\\[0.2cm]
\partial_X^\infty f(\bar{x}) &\mbox{if}\;\lambda=0,\\[0.2cm]
\emptyset &\mbox{if}\;\lambda<0.
\end{array}
\right.
\end{equation}
\item[(iii)] $\partial_X  f(\bar{x})$ is nonempty and compact with
\begin{equation}\label{lpsz-mo}
       {\lip}_X f(\bar{x})= \max_{v\in \partial_X  f(\bar{x})}\|v\|.
       \end{equation}
\end{description}
\end{theorem}
\begin{proof}
As noted at the very beginning of the section, all the results  can be verified by applying Theorem \ref{theo-point-convex} to the profile mapping $E_f$.  Let $Q:=\gph (E_f|_X)$. Clearly,  $Q=\epi(f+\delta_X)$ and $f+\delta_X$ is locally lsc at $\bar{x}$.
So by \cite[Theorem 8.9]{roc}, we have for all $x\in X$,
 \[
N_Q(x, f(x))=\{\lambda(v, -1)\mid \lambda>0, v\in \partial (f+\delta_X)(x)\}\cup\{(v, 0)\mid v\in \partial^{\infty}(f+\delta_X)(x)\},
\]
and hence by definition,
 \begin{equation}\label{prof-co-der}
D^*(E_f|_X)\left(x\mid f(x)\right)(\lambda)=\left\{
\begin{array}{ll}
\lambda \partial (f+\delta_X)(x) &\mbox{if}\;\lambda>0,\\[0.2cm]
\partial^\infty (f+\delta_X)(x) &\mbox{if}\;\lambda=0,\\[0.2cm]
\emptyset &\mbox{if}\;\lambda<0.
\end{array}
\right.
\end{equation}
By Lemma \ref{lem-pro-epi}, we have for all  $x\in X$ and $\alpha>f(x)$,
\begin{equation}\label{zzll}
\gph D^{\rm *prox}(E_f|_X)\left(x\mid \alpha\right)\subset \gph D^*(E_f|_X)\left(x\mid f(x)\right).
\end{equation}

[(a) $\Longrightarrow$ (b)]: By Theorem \ref{theo-point-convex}, there is some $\kappa\geq 0$ such that the inequality
      \[
\|{\rm proj}_{T_X(x)}(v)\| \leq \kappa |\lambda|\quad\forall v\in D^*(E_f|_X)\left(x\mid \alpha\right)(\lambda)
\]
holds for all $(x,\alpha)$ close enough to $(\bar{x}, f(\bar{x}))$ in $Q$ (i.e., $x\in X$ with $f(x)\leq \alpha$).
In view of the  continuity of $f$ at $\bar{x}$ relative to $X$ (due to the Lipschitzian continuity of $f$ at $\bar{x}$ relative to $X$), the above inequality holds in particular for all $x$ close enough to $\bar{x}$ in $X$ with $\alpha=f(x)$. In combining this with the formula (\ref{prof-co-der}), we assert that (b)   holds as the following inequality  holds for  all $x$ close enough to $\bar{x}$ in $X$:
      \[
\|{\rm proj}_{T_X(x)}(v)\| \leq \kappa  \quad\forall v\in \partial (f+\delta_X)(x).
\]
 Note that the convention ${\rm proj}_{T_X(x)}\partial (f+\delta_X)(x):=\emptyset$ is used in (b) for the case that $x\not\in X$.

[(b) $\Longrightarrow$ (a) and (i)]: According to (b),  there are some $\delta>0$ and $\tau>0$ such that
\begin{equation}\label{mtr}
{\rm proj}_{T_X(x)} \partial (f+\delta_X)(x) \subset \tau \B
\end{equation}
holds for all $x\in X$ with $\|x-\bar{x}\|\leq \delta$ and $|f(x)-f(\bar{x})|\leq \delta$.

Let $x\in X$ with $\|x-\bar{x}\|\leq \frac{\delta}{2}$ and $|f(x)-f(\bar{x})|\leq \frac{\delta}{2}$. First  we  show 
\begin{equation}\label{kongzhi-inclusion222}
 \partial^\infty (f+\delta_X)(x)\subset N_X(x).
 \end{equation}
Let $v\in \partial^\infty (f+\delta_X)(x)$.
By the definition of the horizon subgradient, there are sequences $\lambda_k\downarrow 0 $ and  $x_k   \xrightarrow[]{f}  x$ with $x_k\in X$  and  $v_k\in \partial (f+\delta_X)(x_k)$ for all $k$ such that $\lambda_k v_k\to v$. In view of (\ref{mtr}),  the following inequality holds for all $k$ sufficiently large:
\[
 \|{\rm proj}_{T_X(x_k)}(v_k)\|\leq \tau,
\]
or equivalently (as in the proof of Theorem \ref{theo-point-convex}),
\begin{equation}\label{zjls}
\max_{\tilde{w}\in T_X(x_k)\cap \Sph}\langle v_k, \tilde{w}\rangle\leq \tau.
\end{equation}
Let $w\in T_X(x)\cap \Sph$.
As $X$ is closed and convex  (implying that $w$ is a regular tangent vector to $X$ at $x$), it follows from regular tangent cone properties \cite[Theorem 6.26]{roc} that there exists some $w_k\in T_X(x_k)$ such that $w_k\rightarrow w$. In view of (\ref{zjls}), we have for all $k$ sufficiently large,
\[
\left\langle v_k, \frac{w_k}{\|w_k\|}\right\rangle\leq \tau \quad\mbox{and}\quad\left\langle \lambda_kv_k, \frac{w_k}{\|w_k\|}\right\rangle\leq \tau\lambda_k,
\]
implying that $\langle v, w\rangle\leq 0$ and hence  $v\in T_X(x)^*=N_X(x)$.  That is, (\ref{kongzhi-inclusion222}) follows.

Next we suppose by contradiction that (a) is not fulfilled. Then by Theorem \ref{theo-point-convex} again, there exist some $(x_k, \alpha_k) \to (\bar{x}, f(\bar{x}))$ with $x_k\in X$ and $f(x_k)\leq \alpha_k$, and some
\[
v_k\in D^{\rm *prox}(E_f|_X)\left(x_k\mid \alpha_k\right)(\lambda_k)\subset D^*(E_f|_X)\left(x_k\mid f(x_k)\right)(\lambda_k)
\]
 (the inclusion due to (\ref{zzll})) such that  the following inequality holds for all $k$:
      \begin{equation}\label{fanli}
\|{\rm proj}_{T_X(x_k)}(v_k)\| > k |\lambda_k|.
\end{equation}
Clearly, we have \[
\limsup_{k\to +\infty}f(x_k)\leq \limsup_{k\to +\infty}\alpha_k= f(\bar{x})\]
 and hence $f(x_k)\to f(\bar{x})$ (due to $f$ being locally lsc at $\bar{x}$).  In view of (\ref{kongzhi-inclusion222}), we have  for all $k$ sufficiently large, \[
  \partial^\infty (f+\delta_X)(x_k)\subset N_X(x_k) \Longleftrightarrow \|{\rm proj}_{T_X(x_k)}(v)\|=0\quad\forall v\in  \partial^\infty (f+\delta_X)(x_k).
 \]
Thus  for sufficiently large $k$, we deduce from (\ref{prof-co-der}) that $\lambda_k>0$ and hence   $\frac{v_k}{\lambda_k}\in  \partial(f+\delta_X)(x_k)$, for otherwise $\lambda_k=0$ would  imply   $v_k\in \partial^\infty (f+\delta_X)(x_k)$   and hence   $\|{\rm proj}_{T_X(x_k)}(v_k)\|=0$, contradicting to (\ref{fanli}). In combining this with the inequality  (\ref{fanli}), we assert that
      \[
\limsup_{k\rightarrow +\infty} \left\|{\rm proj}_{T_X(x_k)}\left(\frac{v_k}{\lambda_k}\right)\right\|=+\infty,
\]
contradicting to (b). This contradiction indicates that (a) must be fulfilled. As the relative continuity of $f$ is implied by (a), we get (i) immediately from (\ref{kongzhi-inclusion222}).

 [(b) $\Longleftrightarrow$ (d)]: The equivalence follows readily from the definition of the horizon projectional subgradients in Definition \ref{def-proj_subg}.

 When the properties described  in  (\ref{prof-co-der}), (a), (b) and (i) are taken into account, the formula for $D^*_X E_f\left(\bar{x}\mid f(\bar{x})\right)$ in (ii) can be  obtained in a straightforward  way from Definitions \ref{def-proj-outer-norm} and \ref{def-proj_subg}.
 In view of the fact that $f+\delta_X$ is finite and locally lsc at $\bar{x}$, we get from  \cite[corollary 8.10]{roc} that there exists some $\tiny x_k\xrightarrow[]{\;f+\delta_X\;} \bar{x}$ such that $\partial (f+\delta_X)(x_k)\not=\emptyset$.  So the nonemptiness of $\partial_X  f(\bar{x})$ follows from the boundedness of the sequence  ${\rm proj}_{T_X(x_k)}(v_k)$ with $v_k\in \partial (f+\delta_X)(x_k)$. The boundedness of $\partial_X  f(\bar{x})$ follows readily from the local boundedness in (b). That is, $\partial_X  f(\bar{x})$ is nonempty and compact. By the definition of the outer norm, we have
 \begin{equation}\label{waifanshu}
 |D^*_X E_f\left(\bar{x}\mid f(\bar{x})\right)|^+=\max_{v\in \partial_X  f(\bar{x})}\|v\|.
 \end{equation}
 The formula for ${\lip}_X f(\bar{x})$ in (iii) then follows from Theorem \ref{theo-convex-aubin}.

 [(c) $\Longrightarrow$ (b)]: The implication is trivial as ${\rm proj}_{T_X(x)}  \partial (f+\delta_X)(x) \subset \partial_X f(x)$ holds for all $x$ by definition.

 [(a) $\Longrightarrow$ (c)]: By (a), we  have $\lip_Xf(\bar{x})<+\infty$. Then for all $x$ close enough to $\bar{x}$ in $X$,  $f$ is locally Lipschitz continuous at $x$ relative to $X$ with
\[
\max_{v\in \partial_X f(x)}\|v\|= {\lip}_X f(x)\leq {\lip}_Xf(\bar{x}).
 \]
That is,  (c) follows.  This completes the proof.   \end{proof}

In what follows, we consider a proper, lsc, sublinear function $h$ on $\R^n$. It is well-known that there is a unique closed, convex set $D$ in $\R^n$ such that $h=\sigma_D$, i.e., $h$ can be expressed as the support function of $D$.  If $D$ is bounded, then $h$ is finite everywhere, entailing that $h$ is locally Lipschitz continuous everywhere, and in particular,
$${\lip}\,h(0)=\max_{v\in \partial h(0)}\|v\|=\max_{v\in D}\|v\|.$$
However, when $D$ is unbounded, $h$ is not finite everywhere anymore, and in this case, it is more desirable to study the relative Lipschitzian property of $h$ on $\dom h$, which, due to the positive homogeneity of $h$, amounts to the local Lipschitz continuity of $h$ at $0$ relative to $\dom h$.

In the following corollary,   we  apply  Theorem \ref{the-locally-lip-X} to fully  characterize the local Lipschitz continuity of $h$ at $0$ relative to $\dom h$ by describing $\partial_{\dom h}h(0)$ and $\partial^\infty_{\dom h}h(0)$ in terms of all the exposed faces  of $D$ along with   corresponding exposed faces of the horizon cone $D^\infty$.

\begin{corollary}[relative Lipschitzian property of sublinear functions]\label{coro-sublin-lip}
Let $h$ be a proper, lsc, sublinear function on $\R^n$ and let $D$ be the unique closed, convex set in $\R^n$ such that $h=\sigma_D$. For each $x\in (D^\infty)^*$, we denote by
\[
F_{D,x}:=\argmax_{v\in D}\langle v, x\rangle
\]
 the (possibly empty) face  of $D$ exposed by $x$ , and by
\[
F_{D^\infty, x}:=\argmax_{v\in D^\infty}\langle v, x\rangle
\]
 the nonempty  face of $D^\infty$  exposed also by $x$. Then the equality
\begin{equation}\label{guanxi}
(F_{D,x})^\infty=F_{D^\infty, x}
\end{equation}
holds for all $x\in (D^\infty)^*$ with $F_{D,x}\not=\emptyset$.  Moreover, in terms of the faces pairing $F_{D,x}$ and $F_{D^\infty, x}$,  the following properties hold:
\begin{description}
  \item[(a)] $\partial_{\dom h}h(0)$ is nonempty and closed with
\[
\partial_{\dom h}h(0)=\cl\bigcup_{x\in (D^\infty)^*}{\rm proj}_{\tiny (F_{D^\infty, x})^*}\,F_{D,x}\quad\mbox{and}\quad\partial^\infty_{\dom h}h(0)=(\partial_{\dom h}h(0))^\infty.
\]
 In particular,    $\partial_{\dom h}h(0)$ contains   the projections of $D$ on $(D^\infty)^*$, as well as all the bounded exposed faces of $D$.
  \item[(b)]  $h$ is locally Lipschitz continuous at 0 relative to $\dom h$ if and only if $\partial_{\dom h}h(0)$ is bounded.
\item [(c)] We have
\begin{equation}\label{mmoo}
{\lip}_{\dom h}h(0)=\sup_{v\in \partial_{\dom h}h(0)}\|v\|=\displaystyle\sup_{x\in (D^\infty)^*}e(F_{D,x}\,,  \,F_{D^\infty, x}).
\end{equation}
\item[(d)] If $D$ is  polyhedral or in other words $h$ is piecewise linear,   then $\partial_{\dom h}h(0)$ is bounded.
\end{description}
\end{corollary}
\begin{proof} As $D^\infty$ is a closed and convex cone and $D=D+D^\infty$, we have $0\in F_{D^\infty, x}$ and $\langle v, x\rangle =0$ for all $v\in F_{D^\infty, x}$ with $x\in (D^\infty)^*$, and moreover in the case of $F_{D,x}\not=\emptyset$, (\ref{guanxi}) follows directly from the definition of horizon cones.

Clearly, $\dom h$ is a convex cone, not necessarily closed. Let $x\in \dom h$ be given arbitrarily. In view of \cite[Theorem 8.24 and Corollary 8.25]{roc}, we have  $(D^\infty)^*=\cl( \dom h)$, $(\dom h)^*=D^\infty$ and
$$\partial h(x)=
\{v\in D\mid x\in N_D(v)\}=\argmax_{v\in D}\langle v, x\rangle=:F_{D, x}.
$$
 So we have $$N_{\dom h}(x)=(\dom h)^*\cap [x]^\perp=D^\infty \cap [x]^\perp=\argmax_{v\in D^\infty}\langle v, x\rangle=:F_{D^\infty, x},$$
 and hence $T_{\dom h}(x)=(F_{D^\infty, x})^*$. 
Note that   $F_{D, x'}=\emptyset$ whenever $x'\in (D^\infty)^*\backslash \dom h$, and that the convention   ${\rm proj}_M \emptyset :=\emptyset$ is used  for any nonempty set $M$. The closedness of $\partial_{\dom h}h(0)$ and the formulas in (a) follow  readily from the definition of the projectional subgradients (Definition \ref{def-proj_subg}) and the positive homogeneity of $h$. As we have $0\in (D^\infty)^*$, $F_{D,\,0}=D$ and $F_{D^\infty,\, 0}=D^\infty$, $\partial_{\dom h}h(0)$ clearly  contains  the projections of $D$ on $(D^\infty)^*$, and  thus is nonempty. Assume that $F$ is  a bounded exposed face of $D$. Then by definition there is some $x'\in (D^\infty)^*$ such that $F=F_{D, x'}$. From (\ref{guanxi}) it follows that  $F_{D^\infty, x'}=\{0\}$ and $(F_{D^\infty, x'})^*=\R^n$. So $F$ is contained in $\partial_{\dom h}h(0)$ as $
{\rm proj}_{\tiny (F_{D^\infty, x'})^*}\,F_{D,x'}=F_{D,x'}$.

The equivalence in (b) follows directly from Theorem \ref{the-locally-lip-X} if the closedness of $\dom h$ is guaranteed on both sides. To see this, let $x_k\to x^*\not=0$ with $x_k\in \ri (\dom h)\backslash\{0\}$ for all $k$ and we will show $x^*\in \dom h$ for both sides.

First assume that $h$ is locally Lipschitz continuous at 0 relative to $\dom h$. Then we have
\[
\limsup_{k\to +\infty}|h(x_k/\|x_k\|)|=\limsup_{k\to +\infty}\frac{|h(x_k)-h(0)|}{\|x_k-0\|}<+\infty,
\]
which implies by the lower semicontinuity of $h$  that
\[
h(x^*/\|x^*\|)\leq \liminf_{k\to +\infty}h( x_k/\|x_k\|)<+\infty,
\]
and hence that $h(x^*)<+\infty$ or equivalently $x^*\in \dom h$.

Next assume that $\partial_{\dom h}h(0)$ is bounded.  Clearly, $L:=(\dom h)^\perp=D^\infty\cap -D^\infty$ is  the lineality space  of $D$. So  we have $D=(D\cap L^\perp)+L$  and hence  $N_D(v_1+v_2)=N_D(v_1)$  for all $v_1\in D\cap L^\perp$ and $v_2\in L$. As $x_k\in \ri (\dom h)$ is assumed, we have $(F_{D^\infty, x_k})^*=T_{\dom h}(x_k)=L^\perp$ and  $\partial h(x_k)\not=\emptyset$. Let $v_k\in \partial h(x_k)$ or equivalently  $v_k\in D$ with $x_k\in N_D(v_k)$. Then there is some $\tilde{v}_k\in D\cap L^\perp $ such that $x_k\in N_D(\tilde{v}_k)=N_D(v_k)$. Thus we have $\tilde{v}_k\in F_{D, x_k}$ and  $\tilde{v}_k={\rm proj}_{(F_{D^\infty, x_k})^*}(\tilde{v}_k)\in \partial_{\dom h}h(0)$.
As $\partial_{\dom h}h(0)$ is assumed to be bounded,  $\{\tilde{v}_k\}$ is clearly a bounded sequence. By taking a subsequence if necessary, we assume that $\tilde{v}_k\to v^*$. Then we have $v^*\in D$ with $x^*\in N_D(v^*)$ or equivalently $v^*\in \partial h(x^*)$. This implies that $x^*\in \dom h$ as expected.

The first equality in (\ref{mmoo}) follows from (b) and Theorem \ref{the-locally-lip-X} in a straightforward way, while the second one follows from (a), the definition of the excess of a set over another set (see Section 1 for the definition),  and  the fact that the following equalities hold for all $x\in (D^\infty)^*$ and $v\in F_{D, x}$:
\[
d(v, F_{D^\infty, x})=\|v-{\rm proj}_{F_{D^\infty, x}}(v)\|=\|{\rm proj}_{(F_{D^\infty, x})^*}(v)\|.
\]

It remains to show (d). Let $x\in (D^\infty)^*$ be given  arbitrarily.
As $D$ is assumed to be polyhedral,  $D^\infty$ and $F_{D,x}\not=\emptyset$ are also polyhedral  and there exists some bounded set $B_x$ such that $F_{D,x}=B_x+(F_{D,x})^\infty=B_x+F_{D^\infty,x}$ (the second equality due to (\ref{guanxi})). So we have
\[
e(F_{D,x},  \,F_{D^\infty, x}):=\sup_{v\in F_{D, x}}d(v, F_{D^\infty, x})=\sup_{b\in B_x, w\in F_{D^\infty, x}}d(b+w, F_{D^\infty, x})\leq \sup_{b\in B_x}\|b\|<+\infty.
\]
This suggests by (\ref{mmoo}) that  $\partial_{\dom h}h(0)$ is bounded, as every polyhedral set has only finitely many faces. This completes the proof.  \end{proof}

To end this subsection,  we illustrate Corollary \ref{coro-sublin-lip} by two simple examples.

\begin{example}
Consider a sublinear function $h:=\sigma_D$ with $D:=\{(x_1, x_2)^T\mid x_1>0,\,x_2\geq 1/x_1\}$. Clearly,   $D^\infty= \R^2_+$, $(D^\infty)^*=-\R^2_+$ and $D$ does not have one-dimensional face.  Moreover, the projection of $D$ on $(D^\infty)^*$ is  $(0, 0)^T$, and each $(x_1, 1/x_1)$ with $x_1>0$ is a zero-dimensional face of $D$.
 It is clear to see from Corollary \ref{coro-sublin-lip} (a)  that  \[\partial_{\dom h}h(0)=\{(x_1, 1/x_1)^T\mid x_1>0\}\cup \{(0, 0)^T\}.\]  Then by Corollary \ref{coro-sublin-lip} (b), $h$ is not locally Lipschitz continuous at 0 relative to $\dom h$ due to the unboundedness of $\partial_{\dom h}h(0)$. Explicitly, we have
\[
h(x)=\left\{
\begin{array}{ll}
-\sqrt{2x_1x_2} &\mbox{if}\;x_1\leq 0,\,x_2\leq 0,\\
+\infty &\mbox{otherwise},
\end{array}
\right.
\]
and in terms of $x_k:=(0, -1/k)^T$ and $x'_k:=(-1/k^2, -1/k)^T$,
\[
\limsup_{k\to +\infty}\frac{|h(x_k)-h(x'_k)|}{\|x_k-x'_k\|}=\limsup_{k\to +\infty}\sqrt{2k}=+\infty,
\]
implying that $h$ is not locally Lipschitz continuous at 0 relative to $\dom h$. In this case, however,  $\dom h$ is closed and $h$ is  continuous   relative to  $\dom h$.
\end{example}

\begin{example}
Consider a sublinear function $h:=\sigma_{D}$ with $$D:=\left\{(x_1, x_2)^T\mid -1\leq x_1\leq 1,\, x_2\geq 1-\sqrt{1-x_1^2}\right\}.$$
In this case, we have $D^\infty=\{0\}\times \R_+$,  $(D^\infty)^*=\{(x_1, x_2)^T\mid x_2\leq 0\}$, and by Corollary \ref{coro-sublin-lip} (a),
\[
 \partial_{\dom h}h(0)=\left\{(x_1, 0)^T\mid -1\leq x_1\leq 1\right\}\cup \left\{(x_1, x_2)^T\mid -1\leq x_1\leq 1,\, x_2= 1-\sqrt{1-x_1^2}\right\},
\]
which is bounded. Then by Corollary \ref{coro-sublin-lip} (b), $h$ is locally Lipschitz continuous at 0 relative to $\dom h$ with $${\lip}_{\dom h}h(0)=\max_{v\in \partial_{\dom h}h(0)}\|v\|=\sqrt{2}.$$
 Explicitly, we have $h(x)=\|x\|+x_2+\delta_{(D^\infty)^*}(x)$, by virtue of which  we can also verify the local Lipschitzian continuity of $h$ and calculate the constant ${\lip}_{\dom h}h(0)$.
\end{example}

\section{Level-set mappings and structural  subgradients}\label{sec-level-set and ss}

Given a function $f:\R^n\rightarrow \overline{\R}$, a point $\bar{x}\in \R^n$  where $f$ is finite and locally lsc,  and  a vector $\bar{v}\in \R^n$,    the level-set mapping
  \begin{equation}\label{lev-mapping}
  S: \alpha \mapsto \{x\mid f(x)-\langle \bar{v}, x-\bar{x}\rangle\leq \alpha\}
  \end{equation}
   fails to have the Lipschitz-like property at $f(\bar{x})$ for $\bar{x}$ if and only if $\bar{v} \in \partial f(\bar{x})$. This reinterpretation of subgradients has been pointed out in  \cite[Theorem 9.41]{roc} by applying the Mordukhovich criterion via the coderivative
 \[
\left[D^*S\left(f(\bar{x})\mid \bar{x}\right)\right]^{-1}(\lambda)=\left\{
\begin{array}{ll}
\lambda (\partial f(\bar{x})-\bar{v}) &\mbox{if}\;\lambda<0,\\[0.2cm]
-\partial^\infty f(\bar{x}) &\mbox{if}\;\lambda=0,\\[0.2cm]
\emptyset &\mbox{otherwise}.
\end{array}
\right.
\]
Moreover, the graphical modulus of $S$  at $f(\bar{x})$ for $\bar{x}$ can be  given by
\begin{equation}\label{momo}
\lip S(f(\bar{x})\mid \bar{x})=\frac{1}{d(\bar{v}, \partial f(\bar{x}))}.
\end{equation}

One typical circumstance  under which the level set-mapping  $S$ fails in an obvious way to have the Lipschitz-like property  is that the reference point does not belong to the interior of  $\dom S$.  For instance, whenever $f$ is convex and $\bar{v}\in \partial f(\bar{x})$,  we have
\[
\bar{x}\in \arg\min_{x\in \R^n}\{f(x)-\langle \bar{v}, x-\bar{x}\rangle\},
\]
implying that $\dom S=\{\alpha\in \R\mid \alpha\geq f(\bar{x})\}$ and hence that $f(\bar{x})\not\in {\rm int}(\dom S)$.   This motivates us to think of some weaker stability  properties that $S$ may often have, such as the Lipschitz-like property of $S$ relative to
\begin{equation}\label{XXX}
X:=\{\alpha\in \R\mid \alpha\geq f(\bar{x})\}
 \end{equation}
  at $f(\bar{x})$  for $\bar{x}$. It turns out in the sequel that the level set-mapping $S$ fails to have the Lipschitz-like property of $S$ relative to $X$ at $f(\bar{x})$ for $\bar{x}$ if and only if
    \begin{equation}\label{struc-sub}
\bar{v}\in \partial^>_{\bar{v}} f(\bar{x}):=\left\{ \lim_{k\rightarrow +\infty}v_k\mid \exists x_k\rightarrow_f \bar{x},\;\forall k:\,f(x_k)>f(\bar{x})+\langle \bar{v}, x_k-\bar{x}\rangle\,\mbox{and}\,v_k\in \partial f(x_k)\right\},
 \end{equation}
where $\partial^>_{\bar{v}} f(\bar{x})$, called the outer limiting subdifferential set of $f$ at $\bar{x}$ with respect to $\bar{v}$, consisting of outer limiting subgradients with respect to $\bar{v}$ defined using not all nearby $f-$attentive points $x$ unless $f(x)>f(\bar{x})+\langle \bar{v}, x-\bar{x}\rangle$ (implying that $\partial^>_{\bar{v}} f(\bar{x})\subset \partial f(\bar{x})$).
In the case of $\bar{v}=0$, $\partial^>_{\bar{v}} f(\bar{x})$ reduces to the outer limiting subdifferential set (denoted by $\partial^>f(\bar{x})$ without mention of $\bar{v}$), which has been studied extensively in the literature \cite{ac14, io08, knt10,fhko10,io15,ca2, lmy2017, ers2019}.

In the following, we will first provide  formulas for the projectional  coderivative  $D^*_X S\left(f(\bar{x})\mid \bar{x}\right)$ and its outer norm, and then apply Theorem \ref{theo-convex-aubin} to characterize the Lipschitz-like property of $S$ relative to $X$ at $f(\bar{x})$ for $\bar{x}$ via the outer limiting subdifferential set $\partial^>_{\bar{v}} f(\bar{x})$.

\begin{proposition}[projectional coderivative of level-set mappings]\label{ex-stru}
 Given a function $f:\R^n\rightarrow \overline{\R}$, a point $\bar{x}\in \R^n$  where $f$ is finite and locally lsc,  a vector $\bar{v}\in \R^n$, and a  level-set mapping $S$ in the form of (\ref{lev-mapping}),  we have
  \[
\left[D^*_X S\left(f(\bar{x})\mid \bar{x}\right)\right]^{-1}(\lambda)=\left\{
\begin{array}{ll}
\lambda \left(\partial^>_{\bar{v}} f(\bar{x})-\bar{v}\right) &\mbox{if}\;\lambda<0,\\[0.2cm]
-\partial^\infty \tilde{f}(\bar{x})\cup-\pos (\partial \tilde{f}(\bar{x})) &\mbox{if}\;\lambda=0,\\[0.2cm]
\emptyset &\mbox{otherwise},
\end{array}
\right.
\]
and
\[
|D^*_XS\left(f(\bar{x})\mid \bar{x}\right)|^+=\frac{1}{d(\bar{v}, \partial^>_{\bar{v}} f(\bar{x}))},
\]
where $\tilde{f}(x):=\max\{f(x)-\langle\bar{v}, x-\bar{x}\rangle, f(\bar{x})\}$ and $X$ is given by (\ref{XXX}).
 \end{proposition}
\begin{proof} The formula for  $|D^*_X S\left(f(\bar{x})\mid \bar{x}\right)|^+$ follows readily from the formula for $D^*_XS\left(f(\bar{x})\mid \bar{x}\right)$ and the definition of outer norms.

Due to  $f$ being finite and locally lsc at $\bar{x}$, $\tilde{f}$ is also finite and locally lsc at $\bar{x}$. Moreover, we have $\tilde{f}(x)\geq f(\bar{x})$ for all $x$,  and $\partial \tilde{f}(x)=\partial f(x)-\bar{v}$ for all $x$ with $f(\bar{x})<f(x)-\langle\bar{v}, x-\bar{x}\rangle<+\infty$. Let $E:=\gph S|_X$.  Then, $E=\{(\alpha, x)\mid (x, \alpha)\in \epi \tilde{f}\}$. By \cite[Theorem 8.9]{roc}, we have
\begin{equation}\label{piaoliang}
N_E(\tilde{f}(x), x)=\{\lambda(-1, v)\mid \lambda>0, v\in \partial\tilde{f}(x)\}\cup\{(0, v)\mid v\in \partial^{\infty}\tilde{f}(x)\}
\end{equation}
for all $(\tilde{f}(x), x)$ close enough to $(f(\bar{x}), \bar{x})$.

Let $(\lambda, v)\in \R\times \R^n$ be such that  $\lambda\in D^*_XS\left(f(\bar{x})\mid \bar{x}\right)(v)$. By definition, there are some $(\alpha_k, x_k)\rightarrow (f(\bar{x}),\bar{x})$ with $\tilde{f}(x_k)\leq \alpha_k$ and some $(\lambda_k^*, -v_k^*)\in N_{E}(\alpha_k, x_k)$ such that $v_k^*\rightarrow v$ and
${\rm proj}_{T_X(\alpha_k)}(\lambda^*_k)\rightarrow \lambda$. Due to $f(\bar{x})=\tilde{f}(\bar{x})$ and $f(\bar{x})\leq \tilde{f}(x_k)$ for all $k$, we have $\alpha_k\rightarrow \tilde{f}(\bar{x})$ and $\tilde{f}(x_k)\rightarrow \tilde{f}(\bar{x})$.  As we have $(-v_k^*, \lambda_k^*)\in N_{\epi \tilde{f}}(x_k, \alpha_k)$ for all $k$, we have $\lambda^*_k\leq 0$ and hence ${\rm proj}_{T_X(\alpha_k)}(\lambda^*_k)\leq 0$ for all $k$. So we have $\lambda\leq 0$, implying that  $\left[D^*_XS\left(f(\bar{x})\mid \bar{x}\right)\right]^{-1}(\lambda)=\emptyset$ whenever $\lambda>0$.

We now consider the case that $\lambda<0$.  In this case, by taking a subsequence if necessary,   we have ${\rm proj}_{T_X(\alpha_k)}(\lambda^*_k)<0$ for all $k$. This entails that  $\alpha_k>f(\bar{x})$, $T_X(\alpha_k)=\R$  and $\lambda^*_k={\rm proj}_{T_X(\alpha_k)}(\lambda^*_k)<0$ for all $k$.  Let $0<\varepsilon_k<\min\{\alpha_k-f(\bar{x}), -\lambda_k^*\}$ for all $k$. Clearly, $\varepsilon_k\downarrow 0$.  In view of $(-v_k^*, \lambda_k^*)\in N_{\epi \tilde{f}}(x_k, \alpha_k)$ for all $k$, we can find  some
 \[
 (-v'_k, \lambda'_k)\in N^{{\rm prox}}_{\epi \tilde{f}}(x'_k, \alpha'_k)
 \]
  with $(x'_k, \alpha'_k)\in \epi \tilde{f}$,  $\|(x'_k, \alpha'_k)-(x_k, \alpha_k)\|\leq \varepsilon_k$ and $\|(-v'_k, \lambda'_k)-(-v_k^*, \lambda^*_k)\|\leq \varepsilon_k$. Clearly, we have $x'_k\to \bar{x}$ with $\tilde{f}(x'_k)\rightarrow f(\bar{x})$, and  $(v'_k, \lambda'_k)\rightarrow (v, \lambda)$ with $\lambda'_k\leq \lambda_k^*+\varepsilon_k<0$  for all $k$. In view of Lemma \ref{lem-pro-epi}, we have  $\alpha'_k=\tilde{f}(x'_k)$ for all $k$, and hence $\tilde{f}(x'_k)\geq \alpha_k-\varepsilon_k>f(\bar{x})$ and \[(-v'_k, \lambda'_k)\in  N_{\epi \tilde{f}}(x'_k, \tilde{f}(x'_k))\]  for all $k$. By \cite[Theorem 8.9]{roc}, we have
  $v'_k/\lambda'_k\in \partial \tilde{f}(x'_k)=\partial f(x'_k)-\bar{v}$ for all $k$, implying by definition that    $v/\lambda\in \partial^>_{\bar{v}} f(\bar{x})-\bar{v}$. Thus,
  we have
  \begin{equation}\label{bbhang}
  \left[D^*_XS\left(f(\bar{x})\mid \bar{x}\right)\right]^{-1}(\lambda)\subset \lambda (\partial^>_{\bar{v}}f(\bar{x})-\bar{v}).
  \end{equation}
   Conversely, let $v'\in \partial^>_{\bar{v}} f(\bar{x})$. Then by definition,
  there are some $x'_k\rightarrow_f \bar{x}$ and $v'_k\rightarrow v'$ such that $\tilde{f}(x'_k)=f(x'_k)-\langle\bar{v}, x'_k-\bar{x}\rangle>f(\bar{x})$ and
   $v'_k-\bar{v}\in \partial f(x'_k)-\bar{v}=\partial \tilde{f}(x'_k)$ for all $k$. Let $\alpha'_k:=\tilde{f}(x'_k)$ for all $k$.    Clearly, we  have $(\alpha'_k, x'_k)\rightarrow (f(\bar{x}),\bar{x})$,  $\lambda (v'_k-\bar{v})\rightarrow \lambda (v'-\bar{v})$ and ${\rm proj}_{T_X(\alpha'_k)}(\lambda)={\rm proj}_{\R}(\lambda)=\lambda\rightarrow \lambda$ with $(\alpha'_k, x'_k)\in E$ and   $(\lambda, -\lambda (v'_k-\bar{v}))\in N_{E}(\alpha'_k, x'_k)$ for all $k$. That is,  $\lambda (v'-\bar{v})\in \left[D^*_XS\left(f(\bar{x})\mid \bar{x}\right)\right]^{-1}(\lambda)$. The reverse inclusion in (\ref{bbhang}) then follows.

We omit the proof for the case that $\lambda=0$ as it can be obtained in a similar way.  This completes the proof. \end{proof}

By the formulas for the projectional coderivative of level-set mappings and its outer norm presented in Proposition \ref{ex-stru}, we can apply Theorem \ref{theo-convex-aubin}  in a straightforward way to give a characterization  for  the Lipschitz-like property  of the level-set mapping $S$  relative to $X$   at $f(\bar{x})$ for $\bar{x}$ via the structural subgradients defined in (\ref{struc-sub}).

\begin{theorem}[reinterpretation of  structural subgradients]\label{theo-ss-reinter}
Consider a function $f:\R^n\rightarrow \overline{\R}$, a point $\bar{x}\in \R^n$  where $f$ is finite and locally lsc, and a vector $\bar{v}\in \R^n$.
The level-set mapping
 \[
  S: \alpha \mapsto \{x\mid f(x)-\langle \bar{v}, x-\bar{x}\rangle\leq \alpha\}
 \]
has the Lipschitz-like property relative to
\[
X:=\{\alpha\in \R\mid \alpha\geq f(\bar{x})\}
\]
 at $f(\bar{x})$ for $\bar{x}$ if and only if
 \[
 \bar{v}\not\in \partial^>_{\bar{v}} f(\bar{x}).
 \]
 Moreover,
 \[
 {\rm lip}_X S(f(\bar{x})\mid\bar{x})=\frac{1}{d(\bar{v}, \partial^>_{\bar{v}} f(\bar{x}))}.
 \]
 \end{theorem}
\begin{remark}
The computation of the outer limiting subdifferential is not easy  in general. However, to check the condition $\bar{v}\not\in \partial^>_{\bar{v}} f(x)$ may be a much easier job. We note that the authors in \cite{ac14}  gave a concrete
example in electrical circuits where the condition $0\not\in \partial^>_0 f(x)$ could be checkable.
\end{remark}

To end this section, we demonstrate by an example that the structural subgradients in $\partial^>_{\bar{v}} f(\bar{x})$ are often located on the boundary of $\partial f(\bar{x})$, and that the graphical modulus  ${\rm lip}_X S(f(\bar{x})\mid\bar{x})$ can be strictly smaller than  ${\rm lip} S(f(\bar{x})\mid\bar{x})$ even when both are finite.
 \begin{example}\label{ex-euclidean}
Consider the absolute value function  $f(x)=|x|$.  By some direct calculation, we have  $\partial f(0)=[-1, 1]$ and
\[
\partial^>_{\bar{v}} f(0)=\left\{
\begin{array}{ll}
\{-1, 1\}&\mbox{if}\;\bar{v}\in (-1, 1),\\[0.2cm]
\{ -1\}&\mbox{if}\;\bar{v}\in [1, +\infty),\\[0.2cm]
\{1\}&\mbox{if}\;\bar{v}\in (-\infty, -1].\\[0.2cm]
\end{array}
\right.
\]
So we have $\bar{v}\not\in \partial^>_{\bar{v}} f(0)$ for all $\bar{v}\in \R$. Then by Theorem \ref{theo-ss-reinter} or by the definition, we assert that
 for all $\bar{v}\in \R$, the level-set mapping
  \[
  S: \alpha \mapsto \{x\in \R\mid |x|-\bar{v} x\leq \alpha\}
 \]
 has the Lipschitz-like property relative to $\R_+$  at $0$ for $0$ with
 \[
{\lip}_{\R_+} S(0\mid 0)=\left\{
\begin{array}{ll}
\frac{1}{\min\{1-\bar{v}, 1+\bar{v}\}}&\mbox{if}\;\bar{v}\in (-1, 1),\\[0.2cm]
\frac{1}{1+\bar{v}}&\mbox{if}\;\bar{v}\in [1, +\infty),\\[0.2cm]
\frac{1}{1-\bar{v}}&\mbox{if}\;\bar{v}\in (-\infty, -1].\\[0.2cm]
\end{array}
\right.
\]
In contrast, $S$ has the Lipschitz-like property   (without relative to a set) at $0$ for $0$ if and only if $\bar{v}\not\in [-1, 1]$. Moreover, we have
\[
 {\lip}_{\R_+} S(0\mid 0)< \lip S(0\mid 0)=\left\{
\begin{array}{ll}
+\infty &\mbox{if}\;\bar{v}\in [-1, 1],\\[0.2cm]
\frac{1}{\bar{v}-1}&\mbox{if}\;\bar{v}\in (1, +\infty),\\[0.2cm]
\frac{1}{-1-\bar{v}}&\mbox{if}\;\bar{v}\in (-\infty, -1),\\[0.2cm]
\end{array}
\right.
\]
where the equality follows from (\ref{momo}).
 \end{example}
 
\section{Conclusions}

By virtue of a newly-introduced projectional coderivative, we obtained a generalized Mordukhovich criterion for characterizing the Lipschitz-like property of a set-valued mapping relative to a closed and convex set, where the candidate parameter under consideration can be at the boundary of the set. We then applied this criterion to show for an extended real-valued function that its relative Lipschitzian continuity is equivalent to the local boundness of its projectional subdifferential and that for a given vector the Lipschitz-like property of the level-set mapping relative to a closed half line is equivalent to this vector being not in the outer limiting subdifferential. It is worth noting that the projection of the normal cone of the graph of the set-valued mapping onto the tangent cone of the set has played a very important role in our approach.

\end{document}